\documentclass{amsart} 
\usepackage{amsmath,amssymb,amsthm}

\newcommand{\la}{\lambda}
\newcommand\e{\varepsilon}
\newcommand\ep{\varepsilon}
\newcommand\de{\delta}
\newcommand\ga{\gamma}
\renewcommand{\phi}{\varphi}

\newcommand{\R}{{\mathbb R}} 
\newcommand{\Z}{{\mathbb Z}}

\newcommand{\N}{{\mathbb N}} 

\newcommand{\be}{\begin{equation}}
\newcommand{\ee}{\end{equation}}

\newcommand{\co}{\colon}

\newcommand{\Dr}{\Delta^\rho}

\newcommand{\vol}{\operatorname{vol}}
\newcommand{\supp}{\operatorname{supp}}
\newcommand{\diam}{\operatorname{diam}}

\newcommand{\spec}{\operatorname{spec}}

\def\tilde{\widetilde}
\def \bfo {\begin {eqnarray*}}
\def \efo {\end {eqnarray*} }
\def \beq {\begin {eqnarray}}
\def \eeq {\end {eqnarray}}

\def\bra{\langle}
\def\cet{\rangle}

\theoremstyle{definition}
\newtheorem{definition}{Definition}[section]
\newtheorem{notation}[definition]{Notation}
\newtheorem{example}[definition]{Example}

\theoremstyle{remark}

\newtheorem*{remark}{Remark}

\theoremstyle{plain}

\newtheorem{theorem}[definition]{Theorem} 
\newtheorem{lemma}[definition]{Lemma} 
 
\newtheorem{proposition}[definition]{Proposition} 
\newtheorem{corollary}[definition]{Corollary} 

\numberwithin{equation}{section}

\begin{document} 

\title{Spectral stability of metric-measure Laplacians}

\author{Dmitri Burago}                                                          
\address{Dmitri Burago: Pennsylvania State University,                          
Department of Mathematics, University Park, PA 16802, USA}                      
\email{burago@math.psu.edu}                                                     
                                                                                
\author{Sergei Ivanov}
\address{Sergei Ivanov:
St.Petersburg Department of Steklov Mathematical Institute,
Russian Academy of Sciences,
Fontanka 27, St.Petersburg 191023, Russia}
\email{svivanov@pdmi.ras.ru}

\author{Yaroslav Kurylev}                                                          
\address{Yaroslav Kurylev:
Department of Mathematics, University College London, Gower Street,
London, WC1E 6BT, UK}                      
\email{y.kurylev@math.ucl.ac.uk}        

\thanks{The first author was partially supported
by NSF grant DMS-1205597.
The second author was partially supported by
RFBR grants 14-01-00062 and 17-01-00128.
The third author was partially supported by 
EPSRC grant EP/L01937X/1}



\begin{abstract}
We consider a ``convolution mm-Laplacian'' operator
on metric-measure spaces and study its spectral properties.
The definition is based on averaging over small metric balls.
For reasonably nice metric-measure spaces we prove
stability of convolution Laplacian's spectrum
with respect to metric-measure perturbations
and obtain Weyl-type estimates on the number of eigenvalues.
\end{abstract}

\maketitle 

\section{Introduction}
\label{sec:notation}

This paper is motivated by \cite{BIK14} where we
approximate a compact Riemannian manifold
by a weighted graph and show that the spectra of
the Beltrami--Laplace operator on the manifold
and the graph Laplace operator are close to each other.
The goals pursued in this paper are however different. We introduce
an analog of the Laplacian operator for a class of reasonably
nice metric-measure spaces and study its spectral properties. 
This is where this idea comes from: 
The key constructions of \cite{BIK14} can be regarded
as a definition of an operator which approximates
the Beltrami--Laplace operator. 
The definition is based on averaging over small sets.
The construction makes sense for general metric-measure
spaces, which in particular include Riemannian manifolds
and weighted graphs.

In particular, in this paper we show that analogues
of some results from \cite{BIK14} hold for a large class
of metric-measure spaces.
Namely we consider a ``convolution Laplacian'' operator
with a parameter $\rho>0$ (a radius) and prove
that its spectrum enjoys stability under metric-measure approximations.

Recall that a metric-measure space 
is a triple $(X,d,\mu)$ where $(X,d)$ is a metric space
and $\mu$ is a Borel measure on $X$.
All metric spaces in this paper are compact
and all measures are finite.
We denote by $B_r(x)$ 
the metric ball of radius $r$ centered at
a point $x\in X$.

Our main object of study is defined as follows.

\begin{definition}\label{d:rho-laplacian}
Let $X=(X,d,\mu)$ be a metric-measure space and $\rho>0$.
The \textit{$\rho$-Laplacian}
$\Dr_X\co L^2(X)\to L^2(X)$
is defined by
\be \label{1.21.12}
\Dr_X u(x)= \frac{1}{\rho^2\mu(B_\rho(x))} \int_{B_\rho(x)} \big(u(x)-u(y)\big)\, d\mu(y)
\ee
for $u\in L^2(X)$.
\end{definition}

If $X$ is a Riemannian $n$-manifold,
then $\Dr_X$ converges  as $\rho\to 0$ (e.g.\ on smooth functions)
to the Beltrami--Laplace operator multiplied
by the constant $\frac{-1}{2(n+2)}$.
For general metric-measure spaces, it is not clear what should replace
the normalizing constant $\frac1{2(n+2)}$. 
Thus it does not appear in our definition.
For ``good'' metric-measure spaces operators of this type
(more precisely, the corresponding energy functionals)
have meaningful limits as $\rho\to 0$,
called Korevaar-Schoen type energies
(cf.~\cite{Korevaar-Schoen}).
In particular Korevaar-Schoen type energies can be defined
if the space in question satisfies certain measure contraction properties,
see \cite{Strurm98, Kuwae-Shioya03, Kuwae-Shioya08}.

In this paper we study a different type of problems,
namely we consider the operator $\Dr_X$ for a fixed 
``small'' value of $\rho$.
Our goal is to study the spectrum of $\Dr_X$
and its stability properties. 

This extends to the case when $X$ is a discrete space.
In this case all needed geometric data amounts
to weights of points and the information of which 
pairs of points are within distance $\rho$.
This structure is just a weighted graph
(without any lengths assigned to edges)
and the $\rho$-Laplacian  defined by \eqref{1.21.12}
is just the classic weighted graph Laplacian.
The spectral theory of graph Laplacians is a well developed subject,
see e.g.\ \cite{Chung,Sunada}.

In the case when $X$ is a Riemannian manifold, the spectral properties
of $\rho$-Laplacians are studied in \cite{LM} in connection
with random walks on the manifold.
In our previous paper \cite{BIK14} we compare spectra
of weighted graphs approximating a Riemannian manifold $X$
and the spectrum of the Beltrami-Laplace operator $\Delta_X$.
This can be though of as a combination of two steps:
First, one estimates how close the spectra of
$\Dr_X$ and $\Delta_X$ are,
and next, one estimates the difference between the spectra
of $\Dr_X$ and $\Dr_{\Gamma}$ where $\Gamma$ is a graph
approximating~$X$. Here we carry over an analogue
of the second step to general metric-measure spaces.
Our results apply to both ``nice'' metric-measure spaces approximated by
arbitrary ones (Theorem~\ref{t:second}) and to
possibly discrete metric-measure spaces
close to each other (Theorem~\ref{t:stability}).
In the latter theorem we explicitly estimate
the closeness of spectra in terms of the metric-measure
closeness and geometric parameters of the spaces.

It is easy to see that $\Dr_X$ 
is a non-negative self-adjoint operator
with respect to a certain inner product on $L^2(X)$,
see Section \ref{sec:prelim}.
Hence the spectrum of $\Dr_X$ is a subset of $[0,+\infty)$.
Moreover $\spec(\Dr_X)\subset[0,2\rho^{-2}]$.
The spectrum of a bounded self-adjoint operator
divides into the discrete and essential spectra.
The discrete spectrum is the set of
isolated eigenvalues of finite multiplicity
and the essential spectrum is everything else.
It turns out that the essential spectrum of $\Dr_X$,
if nonempty, is the single point $\{\rho^{-2}\}$.
In our set-up we are concerned only with parts
of the spectrum that are substantially below this value.

The following Theorem \ref{t:second} is a not-so-technical implication
of our main results.
It asserts that under suitable conditions
lower parts of $\rho$-Laplacian spectra converge as
the metric-measure spaces in question converge.
Denote by $\la_k(X,\rho)$
the $k$-th smallest eigenvalue of $\Dr_X$
(counting multiplicities).

\begin{theorem}
\label{t:second}
Let a sequence $\{X_n\}$ of compact metric-measure spaces converge
to $X=(X,d,\mu)$ in the sense of Fukaya \cite{Fu}.
Assume that $d$ is a length metric
and $X$ satisfies a version of the Bishop--Gromov inequality:
there is $\Lambda>0$ such that 
\be\label{e:bishop-gromov}
 \frac{\mu(B_{r_1}(x))}{\mu(B_{r_2}(x))} \le \left(\frac{r_1}{r_2}\right)^{\!\!\Lambda}
\ee
for all $x\in X$ and $r_1\ge r_2>0$.
Then 
$$
 \la_k(X,\rho) = \lim_{n\to\infty} \la_k(X_n,\rho)
$$
for all $\rho>0$ and all $k$ such that $\la_k(X,\rho)<\rho^{-2}$.
\end{theorem}

Theorem \ref{t:second} follows from more general
but more technical Theorem \ref{t:stability}
which works for larger classes of spaces
and provides estimates on the rate of convergence.

Now we discuss hypotheses of Theorem \ref{t:second}.
We emphasize that ``niceness'' conditions in Theorem \ref{t:second}
are imposed only on the limit space~$X$. The spaces $X_n$ need not
satisfy them. In particular, $X_n$ can be discrete approximations of~$X$.
Thus, for every ``nice'' space $X$, the spectrum of $\Dr_X$ can be approximated by
spectra of graph Laplacians.

By definition, a metric is a \textit{length metric} if every pair of points
can be connected by a geodesic segment 
realizing the distance between the points.
This condition can be relaxed to an assumption
about intersection of balls, see the BIV condition
in Definition \ref{d:BIV}.

The classic Bishop--Gromov inequality deals with volumes of balls
in Riemannian manifolds with Ricci curvature bounded from below.
It implies \eqref{e:bishop-gromov}
with $\Lambda$ depending on the dimension of the manifold, its diameter,
and the lower bound for Ricci curvature.
(In the case of non-negative Ricci curvature $\Lambda$ is just equal to the dimension.)
The Bishop--Gromov inequality holds for spaces with generalized
Ricci curvature bounds in the sense of Lott--Sturm--Villani \cite{Sturm06,Sturm06-II,LV}.
Other classes of spaces satisfying \eqref{e:bishop-gromov}
include Finsler manifolds, dimensionally homogeneous polyhedral spaces, 
Carnot groups, etc.
We also note that \eqref{e:bishop-gromov} follows from
measure contraction properties
used in \cite{Strurm98, Kuwae-Shioya03, Kuwae-Shioya08, Sturm06-II}.

The Fukaya convergence combines the Gromov--Hausdorff
convergence of metric spaces and weak convergence of measures.
See Definition \ref{d:fukaya} for details.
Beware of the fact that,
unlike most definitions used in this paper,
the Fukaya convergence is sensitive to open sets of zero measure.
See the example in Section~\ref{sec:zero-limit-measure}
for an illustration of this subtle issue.

Actually in this paper we use another notion of metric-measure
approximation which is more suitable to the problem. 
It allows us to obtain nice estimates 
on the difference of eigenvalues of $\rho$-Laplacians 
of close metric-measure spaces.

\subsection*{Structure of the paper}
In Section \ref{sec:prelim} we introduce some notation
and collect basic facts about $\rho$-Laplacians.
In Section \ref{sec:examples} we discuss some examples.

In Section \ref{sec:wasserstein} we introduce a notion
of ``closeness'' of metric-measure spaces,
which we call $(\ep,\de)$-closeness.
Loosely speaking, metric-measure spaces $X$ and $Y$ 
are $(\ep,\de)$-close if $Y$ is a result 
of imprecise measurements in $X$ where distances
are measured with a small additive inaccuracy $\ep$ and volumes are
measured with a small relative inaccuracy~$\de$.
The formal definition is a combination of Gromov--Hausdorff distance
and a ``relative'' version of Prokhorov distance between measures.
The main results of Section \ref{sec:wasserstein}
characterize $(\ep,\de)$-closeness in terms
of measure transports and Wasserstein distances.

In Section~\ref{sec:stability} we prove Theorem \ref{t:stability}
which is a quantitative version of Theorem~\ref{t:second}.
It asserts that, if metric-measure spaces
$X$ and $Y$ are $(\ep,\de)$-close and satisfy certain conditions, 
then the lower parts of the spectra of their $\rho$-Laplacians are also close.
The conditions in Theorem \ref{t:stability}
can be thought of as ``discretized'' version of those from Theorem~\ref{t:second}.

In Section \ref{sec:TXY} 
we give a direct construction of a map between $L^2(X)$ and $L^2(Y)$
realizing the spectral closeness in Theorem \ref{t:stability}.
The results of Section~\ref{sec:TXY}
complement Theorem \ref{t:stability}
but they are not used in its proof.

In Section \ref{sec:weyl} we 
obtain Weyl-type estimates for the number of eigenvalues
in an interval $[0,c\rho^{-2}]$ where $c<1$ is a suitable constant.
See Theorems \ref{t:ess-spectrum} and \ref{t:many-eigenvalues}.
For a Riemannian manifold our estimates are of the same order
as those given by Weyl's asymptotic formula for
Beltrami--Laplace eigenvalues.
However our estimates are formulated in terms of packing numbers
rather than the dimension and total volume.

\subsection*{Acknowledgement}
We are grateful to Y.~Eliashberg, L.~Polterovich,
and a number of other mathematicians
for pointing out weaknesses in preliminary versions of the paper.
We did our best in fixing these issues. We are grateful to F.~Galvin
and the anonymous referee for bringing our attention
to some references.

\section{Preliminaries}
\label{sec:prelim}

In the sequel we abbreviate compact metric-measure spaces 
with finite Borel measures as ``mm-spaces''.
We use notation $d_X$ and $\mu_X$ for the metric and measure
of a mm-space $X$.
In some cases we consider semi-metrics,
that is, distances are allowed to be zero.
All definitions apply to semi-metrics with no change.

To simplify computations and incorporate
constructions from \cite{BIK14} into the present set-up,
we introduce weighted $\rho$-Laplacians.
Let $X=(X,d,\mu)$ be a mm-space and $\phi\co X\to\R_+$
a positive measurable function bounded away from 0 and $\infty$
on the support of $\mu$.
We call $\phi$ the \textit{normalizing function}.
We define a \emph{weighted $\rho$-Laplacian} 
$\Dr_\phi$ by
$$
\Dr_\phi u(x)= 
\frac{1}{\phi(x)} \int_{B_\rho(x)} \bigl(u(x)-u(y)\bigr)\, d\mu(y) .
$$
We regard $\Dr_\phi$ as an operator on $L^2(X)$.
Note that this operator does not change if one replaces
$X$ by the support of its measure.

Definition \ref{d:rho-laplacian} corresponds to
the normalizing function $\phi(x)=\rho^2\mu(B_\rho(x))$.
Due to compactness of $X$, 
this function is bounded away from 0 and $\infty$ 
on the support of $\mu$.

The operator $\Dr_\phi$ is self-adjoint on $L^2(X,\phi\mu)$
where $\phi\mu$ is the measure with density $\phi$ w.r.t.~$\mu$.
Indeed, for $u,v\in L^2(X)$ we have
\begin{align*}
 \langle \Dr_\phi u, v\rangle_{L^2(X,\phi\mu)} 
 &= \int_X \phi(x) v(x) \frac1{\phi(x)} \int_{B_\rho(x)} \bigl(u(x)-u(y)\bigr)\, d\mu(y) d\mu(x) \\
 &= \iint_{d(x,y)<\rho} v(x)\big(u(x)-u(y)\bigr)\, d\mu(y) d\mu(x)
\end{align*}
and the right-hand side is clearly symmetric in $u$ and $v$.
The corresponding Dirichlet energy form
$$
 D^\rho_X(u) = \langle\Dr_\phi u, u\rangle_{L^2(X,\phi\mu)}
$$
does not depend on $\phi$ and is given by
\be\label{e:dirichlet}
D^\rho_X(u) = \frac12 \iint_{d(x,y) <\rho} \bigl(u(x)- u(y)\bigr)^2 \, d\mu(x)  d\mu(y).
\ee
Note that the Dirichlet form is non-negative.

When dealing with $\rho$-Laplacians from Definition \ref{d:rho-laplacian},
that is when $\phi(x)=\rho^2\mu(B_\rho(x))$, 
we denote the measure $\phi\mu$ by $\mu^\rho$.
We denote the scalar product and norm in $L^2(X,\mu^\rho)$ 
by $\langle\cdot,\cdot\rangle_{X^\rho}$ and $\|\cdot\|_{X^\rho}$, resp.
That is,
\begin{gather}
\label{e:murho}
d\mu^\rho(x)/d\mu(x)=\rho^2\mu(B_\rho(x)) , \\
\label{e:L2rho}
\bra u, v \cet_{X^\rho}= \rho^2 \int_X \mu(B_\rho(x)) u(x) v(x)\, d\mu(x) , \\
\label{e:L2norm}
\|u\|^2_{X^\rho}= \rho^2 \int_X \mu(B_\rho(x)) u(x)^2\, d\mu(x) .
\end{gather}
The norm of $\Dr_X$ in $L^2(X,\mu^\rho)$ is bounded by $2\rho^{-2}$.
Indeed,
\begin{align*}
 D_X^\rho(u) 
 &= \frac12 \iint_{d(x,y)<\rho} (u(x)-u(y))^2 \,d\mu(x)d\mu(y) \\
 &\le \iint_{d(x,y)<\rho} (u(x)^2+u(y)^2) \,d\mu(x)d\mu(y) \\
 &= 2\int_X \mu(B_\rho(x)) u(x)^2 \,d\mu(x) = 2\rho^{-2} \|u\|^2_{X^\rho} .
\end{align*}
Thus the spectrum of $\Dr$ is contained in $[0,2\rho^{-2}]$.

The $\rho$-Laplacian $\Dr_X$ can be rewritten in the form $\Dr_X u = \rho^{-2} u - A u$ where
$$
 Au(x) = \frac{1}{\rho^2\mu(B_\rho(x))} \int_{B_\rho(x)} u(y)\, d\mu(y) .
$$
Observe that $A$ is an integral operator with a bounded kernel. 
Hence it is a compact operator on $L^2(X)$.
It follows that the essential spectrum of $\Dr_X$ is the same as that
of the operator $u\mapsto\rho^{-2} u$.
Namely it is empty if $L^2(X)$ is finite-dimensional and
the single point $\{\rho^{-2}\}$ otherwise.

A similar argument shows that the essential spectrum of $\Dr_\phi$
is located between the infimum and supremum of the function
$x\mapsto \mu(B_\rho(x))/\phi(x)$.

\begin{notation}
\label{n:lambda}
Let $\la_\infty=\la_\infty(X,\rho,\phi)$ be
the infimum of the essential spectrum of $\Dr_\phi$.
If there is no essential spectrum
(that is, if $L^2(X)$ is finite-dimensional),
we set $\la_\infty=\infty$.
For every $k\in\N$ we define $\la_k=\la_k(X,\rho,\phi)\in[0,+\infty]$ as follows.
First let $0=\la_1\le\la_2\le\dots$ be the eigenvalues of $\Dr_\phi$
(with multiplicities)
which are smaller than $\la_\infty$. 
If there are only finitely many of such eigenvalues, 
we set
$\la_k=\la_\infty$ for all larger values of~$k$.

We abuse the language and refer to $\la_k(X,\rho,\phi)$ as
the \emph{$k$-th eigenvalue} of $\Dr_\phi$ even though
it may be equal to $\la_\infty$.

For the $\rho$-Laplacian $\Dr_X$ 
we drop $\phi$ from the notation and denote the $k$-th eigenvalue by $\la_k(X,\rho)$.
\end{notation}

By the standard Min-Max Theorem, for every $k\in\N$ we have
\be \label{e:minmax-phi}
\la_k(X,\rho,\phi) =\inf_{H^k} \sup_{u \in H^k\setminus\{0\}}
\left(\frac{ D^\rho_X(u)}{\|u\|_{L^2(X,\phi\mu)}^2}\right)
\ee
and in particular
\be \label{e:minmax}
\la_k(X,\rho) =\inf_{H^k} \sup_{u \in H^k\setminus\{0\}}
\left(\frac{ D^\rho_X(u)}{\|u\|_{X^\rho}^2}\right)
\ee
where the infima are taken over all $k$-dimensional
subspaces $H^k$ of $L^2(X)$.
This formula is our main tool for eigenvalue estimates.
We emphasize that it holds in both cases $\la_k<\la_\infty$ and $\la_k=\la_\infty$.

As an immediate application, we observe that the eigenvalues are
stable with respect to small relative changes of the normalizing
function and measure. 
If $\mu_1$ and $\mu_2$ are measures on $X$ satisfying
$a\mu_1\le\mu_2\le b\mu_1$ where $a$ and $b$ are positive constants,
then for the corresponding mm-spaces $X_1=(X,d,\mu_1)$ and $X_2=(X,d,\mu_2)$
we have
\be\label{e:la-change-mu}
 \frac{a^2}{b^2}
 \le \frac{\la_k(X_2,\rho)}{\la_k(X_1,\rho)}
 \le \frac{b^2}{a^2}
\ee
for every $k\in\N$.
This follows from \eqref{e:minmax} and
the inequalities 
\begin{gather*}
a^2\le D^\rho_{X_2}(u)/D^\rho_{X_1}(u)\le b^2 , \\
a^2\le \|u\|^2_{X_2^\rho}/\|u\|^2_{X_1^\rho}\le b^2 ,
\end{gather*}
which hold for all $u\in L^2(X)$.
Note that multiplying the measure by a constant does not
change the $\rho$-Laplacian.

For any two normalizing 
functions $\phi_1$ and $\phi_2$ \eqref{e:minmax-phi} implies that
\be\label{e:la-change-phi}
 \inf_{x\in X}\frac{\phi_1(x)}{\phi_2(x)}
 \le \frac{\la_k(X,\rho,\phi_2)}{\la_k(X,\rho,\phi_1)} 
 \le \sup_{x\in X}\frac{\phi_1(x)}{\phi_2(x)}.
\ee
For nice spaces such as Riemannian manifolds, the volume of small $\rho$-balls
is almost constant as a function of the center of the ball.
In such cases one can consider a weighted $\rho$-Laplacian with a constant
normalizing function and conclude that its spectrum is close to that of $\Dr_X$
(cf.\ Section \ref{subsec:riem}).

\section{Examples}
\label{sec:examples}

\subsection{Riemannian manifolds}
\label{subsec:riem}
The paper \cite{BIK14} deals with the case 
of $X$ being a closed Riemannian $n$-manifold $M$
or a discrete approximation of~$M$.
In the terminology of Section \ref{sec:prelim},
the object studied in \cite{BIK14} is a
weighted $\rho$-Laplacian with constant normalization function
$\phi(x)=\phi_\rho:=\frac{\nu_n\rho^{n+2}}{2n+4}$.
Here $\nu_n$ is the volume of the unit ball in $\R^n$.
As $\rho\to 0$, we have $\mu(B_\rho(x))\sim\nu_n\rho^n$
uniformly in $x\in M$.
Hence $\phi_\rho/\rho^2\mu(B_\rho(x))\to \frac1{2n+4}$.
Thus, by \eqref{e:la-change-phi}, the spectrum of $\Dr_X$
is close to that of $\Dr_\phi$ multiplied by $\frac1{2n+4}$.

The results of \cite{LM} imply that
the spectrum of $\Dr_X$, where $X$ is a Riemannian manifold, converges as $\rho\to 0$ to 
the Beltrami--Laplace spectrum multiplied by $\frac1{2n+4}$.
In \cite{BIK14} similar convergence is shown for graph Laplacians
arising from discrete approximations of a Riemannian manifold.
Theorem \ref{t:second} generalizes this result.


Note that the scalar product
$\langle\cdot,\cdot\rangle_{X^\rho}$ and Dirichlet form $D^\rho_X$
tend to 0 as $\rho\to 0$.
To make them comparable with the Riemannian counterparts
one multiplies them by $\rho^{-n-2}$.

\subsection{Finsler manifolds}
Let $X$ be a closed Finsler manifold $M$ with smooth
and quadratically convex Finsler structure.
First recall that there are many reasonable notions of volume
for Finsler manifolds, see e.g.~\cite{Thompson}. 
Different volume definitions obviously lead
to different $\rho$-Laplacians. 
Still the issues we study in this paper are
not sensitive to the choice of volume.

Consider a tangent space $V=T_xM$ at a point $x\in M$.
It is equipped with a norm $\|\cdot\|=\|\cdot\|_x$
which is the restriction of the Finsler structure.
Let $B$ be the unit ball of $\|\cdot\|$.
There is a unique ellipsoid $E\subset V$
such that for every quadratic form the integrals
of it over $B$ and $E$ coincide.
Rescaling $E$ by a suitable factor
(depending on the chosen Finsler volume definition)
and regarding the resulting ellipsoid as the unit
ball of a Euclidean metric,
one obtains a Euclidean metric $|\cdot|$ on $V$
whose $\rho$-Laplacian coincides with that of $\|\cdot\|$
on the set of quadratic forms on $V$.

Applying this construction to every $x\in M$ one obtains
a family of quadratic forms on the tangent spaces
thus defining a Riemannian metric on~$M$. 
It is very likely that the spectra of $\rho$-Laplacians
of the Finsler metric converge as $\rho\to 0$
to the Beltrami--Laplace spectrum of this Riemannian metric.

\subsection{Piecewise Riemannian polyhedra}
Let $X$ be a finite simplicial complex whose faces
are equipped with Riemannian metrics which
agree on the intersections of faces.
First assume that $X$ is dimensionally homogeneous
of dimension $n$.
In this case one can mostly follow the analysis
of the Riemannian case.
The difference is that,
due to boundary terms, the Riemannian Dirichlet energy
$\int_X\|du\|^2$ is not always equal
to $\langle \Delta u,u\rangle$ where $\Delta$
is the Beltrami--Laplace operator.
They are however equal on the subspace of functions
satisfying Kirchhoff's condition.
This condition says that, at every point in
every $(n-1)$-dimensional face, the sum of normal
derivatives in the adjacent $n$-dimensional faces
equals~0.
For instance, if $X$ is a manifold with boundary,
this boils down to the Neumann boundary condition.
It is plausible that the spectra of $\Dr_X$
converge as $\rho\to 0$ to the spectrum
of Beltrami--Laplace operator with Kirchhoff's condition.

The problem can also be studied for polyhedral spaces
with varying local dimension. 
For instance, consider a two-dimensional membrane
with a one-dimensional string attached.
One can equip this space with a measure which
is one-dimensional on the string and 
two-dimensional on the membrane.
Unlike the previous examples,
we cannot apply our results to this example
because it does not satisfy the doubling condition.
It is violated near the point where the string
is attached to the membrane.
It is rather intriguing if 
Theorem~\ref{t:second}
still holds in this situation.

\subsection{Disappearing measure support}
\label{sec:zero-limit-measure}
The following example shows that one has
to be careful with limits of mm-spaces if
the limit measure does not have full support.

Let $X$ be a disjoint union of two compact Riemannian manifolds
$M_1$ and $M_2$. 
Define a distance $d$ on $X$ as follows: In each component
it is the standard Riemannian distance, and the distance between the components
is a large constant.
For each $t\ge 0$ define a measure $\mu_t$ on $X$
by $\mu_t = \vol_{M_1}+t\vol_{M_2}$
where $\vol_{M_i}$, $i=1,2$, are Riemannian volumes
on the components. 
Then $\mu_t$ weakly converge to $\mu_0=\vol_{M_1}$ as $t\to 0$.

For every $t>0$, locally constant functions form a two-dimensional
subspace in $L^2(X)$.
Hence the zero eigenvalue of $\Dr_{X_t}$ has multiplicity~2.
Thus $\la_2(X_t,\rho)=0$ for all $t>0$.
On the other hand, $\la_2(X_0,\rho)>0$ since
the $\rho$-Laplacian of $X_0$ is the same as
that of the component $M_1$.
Thus $\la_1(M_0)\ne\lim_{t\to0}\la_1(M_t)$.

A formal reason for the failure of Theorem \ref{t:second}
in this example
is that the Bishop--Gromov condition \eqref{e:bishop-gromov}
is not satisfied.
Another issue is that $d$ is not a length metric.
The latter can be fixed by connecting $M_1$ and $M_2$
by a long segment and taking the induced intrinsic metric.

\section{Relative Prokhorov and Wasserstein closeness}
\label{sec:wasserstein}

This section is devoted to the notion of $(\ep,\de)$-closeness
that we use in our spectrum stability results.
This notion is introduced in Definition \ref{d:mm-close}.
The main results of this section are
Proposition \ref{p:prokhorov-vs-wasserstein} 
and Corollary \ref{c:mm-wasserstein}
which characterize $(\ep,\de)$-closeness 
in terms of measure transport.

We use the following notation.
For a metric space $(X,d)$ and a set $A\subset X$ and $r\ge 0$, 
we denote by $A^r$ the closed $r$-neighborhood of $A$.
That is, $A^r=\{x\in X:\, d(x,A)\le r\}$.

\begin{definition}[relative Prokhorov closeness] \label{d:relative-prokhorov}
Let $Z$ be a metric space, $\mu_1$, $\mu_2$ finite Borel measures on $Z$,
and $\ep,\de\ge 0$.
We say that $\mu_1$ and $\mu_2$
are \emph{relative $(\e,\delta)$-close} if for every Borel set $A \subset Z$,
$$
e^\delta \mu_1(A^\ep) \geq \mu_2(A)
\quad\text{and}\quad
e^\delta \mu_2(A^\ep) \geq \mu_1(A)
.
$$
\end{definition}

This definition is similar to that of Prokhorov's distance 
on the space of measures \cite{Prokhorov}.
The crucial difference is that we use multiplicative
corrections rather than additive ones.

The topology arising from Definition \ref{d:relative-prokhorov}
is stronger than the standard weak topology on the space of measures on $X$.
If however we restrict ourselves to the subspace of measures
with full support, then the topologies are the same.

We combine Definition \ref{d:relative-prokhorov}
with the notion of Gromov--Hausdorff (GH) distance
analogously to the definition of Gromov--Wasserstein distances
as in e.g.\ \cite{Sturm06-I}.
Recall that metric spaces $(X,d_X)$ and $(Y,d_Y)$ are $\ep$-close
in the GH distance iff
the disjoint union $X\sqcup Y$ can be equipped with a (semi-)metric $d$
extending $d_X$ and $d_Y$ and such that $X$ and $Y$
are contained in the $\ep$-neighborhoods of each other
with respect to~$d$.
For discussion of GH distance see e.g.~\cite{BBI}.

\begin{definition}
\label{d:mm-close}
Let $\ep,\de\ge 0$.
We say that mm-spaces $X=(X,d_X,\mu_X)$ and $Y=(Y,d_Y,\mu_Y)$
are \emph{mm-relative $(\ep,\de)$-close} if there exists 
a semi-metric $d$ on $X\sqcup Y$
extending $d_X$ and $d_Y$ and such that $\mu_X$ and $\mu_Y$
are relative $(\ep,\de)$-close in $(X\sqcup Y,d)$ 
in the sense of Definition \ref{d:relative-prokhorov}.

In the sequel we abbreviate ``mm-relative $(\ep,\de)$-close'' to just $(\ep,\de)$-close.
\end{definition}

Observe that, if the measures have full support, 
then $(\ep,\delta)$-closeness of mm-spaces
$(X,d_X,\mu_X)$ and $(Y,d_Y,\mu_Y)$ implies that
the metric spaces $(X,d_X)$ and $(Y,d_Y)$ are 
$\ep$-close in the sense of Gromov--Hausdorff distance.

The following example motivated Definition \ref{d:mm-close}
as well as a number of other definitions and assumptions in this paper.

\begin{example}[discretization, cf.~\cite{BIK14}]
\label{x:discretization}
Let $X$ be a mm-space and $Y$ a finite $\ep$-net in $X$.
We can associate a small basin in $X$ to every point of $Y$
and move all measure from each basin to its point.
More precisely, there is a partition of $X$ into measurable
sets $V_y$, $y\in Y$, such that each $V_y$ is contained
in the ball $B_\ep(y)$. We assign the weight equal to $\mu_X(V_y)$
to each $y$ thus defining a measure $\mu_Y$ on $Y$.
If we regard $\mu_Y$ as a measure on $X$, 
then it is relative $(\ep,0)$-close to $\mu_X$
in the sense of Definition \ref{d:relative-prokhorov}.
We can also regard $Y$ equipped with $\mu_Y$
as a separate mm-space. 
Then it is $(\ep,0)$-close to $X$ in the sense of Definition \ref{d:mm-close}.

Now consider a result of some ``measurement errors'' in~$Y$.
Namely, let $Y'=(Y,d_Y',\mu_Y')$ be a mm-space with the same point set $Y$ 
and such that $|d_Y'-d_Y^{}|<\ep$ and $e^{-\de}\le \mu_Y'/\mu_Y^{}\le e^\de$.
Then $Y'$ is $(2\ep,\de)$-close to~$X$.
\end{example}

Now we show that Fukaya convergence (used in Theorem \ref{t:second})
implies convergence with respect to $(\ep,\de)$-closeness,
provided that the limit measure has full support.
Recall that the Fukaya convergence is defined as follows.

\begin{definition}[cf.~{\cite[(0.2)]{Fu}}]\label{d:fukaya}
A sequence $X_n=(X_n,d_n,\mu_n)$ of mm-spaces
converges to a mm-space $X=(X,d,\mu)$ in the sense of Fukaya
if the following holds.
There exist 
a sequence $\sigma_n\to 0$ of positive numbers
and a sequence $f_n\co X_n\to X$ of measurable maps
such that
\begin{enumerate}
\item $f_n(X_n)$ is an $\sigma_n$-net in $X$;
\item $|d(f_n(x),f_n(y))-d_n(x,y)| < \sigma_n$ for all $x,y\in X_n$;
\item the  push-forward measures $(f_n)^{}_*\mu_n$ weakly converge to~$\mu$.
\end{enumerate}
\end{definition}

\begin{proposition}
\label{p:fukaya}
Let $X_n$ converge to $X$ in the sense Fukaya
and assume that $\mu_X$ has full support.
Then there exist sequences $\ep_n,\de_n\to 0$
such that $X_n$ is $(\ep_n,\de_n)$-close to~$X$
for all $n$.
\end{proposition}

\begin{proof}
Let $X$, $X_n$, $\sigma_n$, $f_n$ be as above.
The existence of $f_n$ implies
that $X_n$ is $2\sigma_n$-close to $X$
in the GH distance, 
see e.g.\ \cite[Cor.~7.3.28]{BBI}.
Moreover there is a metric $d_n'$
on the disjoint union $X\sqcup X_n$
such that $d_n'$ extends $d\cup d_n$ and
\be\label{e:fukaya3}
d_n'(x,f_n(x)) \le \sigma_n
\ee
for all $x\in X_n$.

It suffices to prove that for every $\ep,\de>0$
the spaces $X_n$ eventually get $(\ep,\de)$-close to~$X$.
Fix $\ep$ and $\de$.
Let $\nu=\nu(\ep,\de)>0$ be so small that
$$
 (e^\de-1)(\mu(B_{\ep/3}(x))-\nu) \ge \nu
$$
for all $x\in X$. 
Such $\nu$ exists since the measures 
of $(\ep/3)$-balls in $X$ are bounded away from~0.
This is where we use the assumption that $\mu$ has full support.

Since $(f_n)_*^{}\mu_n$ weakly converges to $\mu$,
the Prokhorov distance between $(f_n)_*^{}\mu_n$ and $\mu$
tends to~0, see \cite{Prokhorov}. 
This implies that for all sufficiently large~$n$ we have
\begin{gather}
\label{e:fukaya1}
 \mu(A^{\ep/3}) + \nu >  \mu_n(f_n^{-1}(A)) , \\
\label{e:fukaya2}
 \mu_n(f_n^{-1}(A^{\ep/3})) + \nu > \mu(A)
\end{gather}
for every Borel set $A\subset X$.

Now consider the disjoint union $Z_n=X\sqcup X_n$
equipped with the metric $d_n'$.
The measures $\mu$ and $\mu_n$ can be regarded
as measures on~$Z_n$.
If $\sigma_n<\ep/3$ then by \eqref{e:fukaya3} we have
$f_n^{-1}(A^{\ep/3}) \subset A^\ep$ for all $A\subset X$
and $(f_n(B))^{\ep/3}\subset B^\ep$ for all $B\subset X_n$.
Here the neighborhoods are taken in $(Z_n,d_n')$.
These inclusions along with \eqref{e:fukaya1} and \eqref{e:fukaya2}
imply that
\begin{gather}
\label{e:fukaya4}
 \mu(A^\ep)+\nu > \mu_n(A), \\
\label{e:fukaya5}
 \mu_n(A^\ep)+\nu > \mu(A)
\end{gather}
for every Borel set $A\subset Z_n$ provided that $n$ is large enough.

Let $A\subset Z_n$ be a nonempty set and $\sigma_n<\ep/6$. 
Then there exists $x\in X$ such that
$A^\ep$ contains the ball $B_{5\ep/6}(x)$.
This fact is trivial if $A\cap X\ne\emptyset$,
otherwise it follows from \eqref{e:fukaya3}.
Let $D=B_{\ep/3}(x)\cap X$. 
By \eqref{e:fukaya3} we have
$$
 f_n^{-1}(D^{\ep/3})\subset f_n^{-1}(B_{2\ep/3}(x))\subset B_{5\ep/6}(x) \subset A^\ep .
$$
Therefore
\be\label{e:fukaya7}
 \mu_n(A^\ep) \ge \mu_n(f_n^{-1}(D^{\ep/3})) > \mu(D)-\nu 
\ee
by \eqref{e:fukaya2}.
Since $D$ is an $(\ep/3)$-ball in $X$, 
by the definition of $\nu$ we have
$$
 (e^\de-1)(\mu(D)-\nu) \ge \nu .
$$
This and \eqref{e:fukaya7} imply that $(e^\de-1)\mu_n(A^\ep)\ge\nu$
and therefore
\be\label{e:fukaya8}
 e^\de \mu_n(A^\ep) \ge \mu_n(A^\ep)+\nu > \mu(A)
\ee
by \eqref{e:fukaya5}.
Similarly, since $D\subset A^\ep$, we have $\mu(A^\ep)\ge\mu(D)$. 
This inequality and \eqref{e:fukaya4} imply that
\be\label{e:fukaya9}
 e^\de \mu(A^\ep) \ge \mu_n(A)
\ee
in the same way as \eqref{e:fukaya7} and \eqref{e:fukaya5}
imply \eqref{e:fukaya8}.
Now \eqref{e:fukaya8} and \eqref{e:fukaya9} imply
that $X_n$ and $X$ are $(\ep,\de)$-close.
The proposition follows.
\end{proof}

Now we reformulate $(\ep,\de)$-closeness in terms of measure transport.
Recall that a \textit{measure coupling} 
(or a \emph{measure transportation plan}) between
measure spaces $(X,\mu_X)$ and $(Y,\mu_Y)$
is a measure $\ga$ on $X\times Y$ whose marginals
on $X$ and $Y$ coincide with $\mu_X$ and $\mu_Y$, resp.
The \emph{marginals} are push-forwards of $\ga$ by
the coordinate projections from $X\times Y$
to the factors.
Obviously a measure coupling exists if and only if
$\mu_X(X)=\mu_Y(Y)$.

In our set-up $X$ and $Y$
are compact subsets of a metric space $(Z,d)$
and all measures are finite Borel.
In this case $\mu_X$ and $\mu_Y$ can be regarded as measures on $Z$
and, assuming that $\mu_X(X)=\mu_Y(Y)$, one defines the $L^\infty$-Wasserstein
distance $W_\infty(\mu_X,\mu_Y)$
as the minimum of all $\ep\ge 0$ such that there exists
a coupling $\ga$ between $\mu_X$ and $\mu_Y$ such that
$d(x,y)\le\ep$ for $\ga$-almost all pairs $(x,y)\in X\times Y$.
(The minimum exists due to the weak compactness of the space of measures.)
For discussion of Wasserstein distances, see e.g.~\cite{Villani}.

\begin{proposition}[approximate coupling]
\label{p:prokhorov-vs-wasserstein}
Let $Z$ be a compact metric space and $\mu_X$, $\mu_Y$
finite Borel measures on $Z$.
Then the following two conditions are equivalent:
\begin{enumerate}
 \item[(i)] 
$\mu_X$ and $\mu_Y$ are relative $(\ep,\de)$-close
(see Definition \ref{d:relative-prokhorov});
 \item[(ii)]
There exist measures $\tilde\mu_X$ and $\tilde\mu_Y$ on $Z$ such that
$$
e^{-\delta}\mu_X \le \tilde\mu_X \le \mu_X, 
\qquad 
e^{-\delta}\mu_Y \le \tilde\mu_Y \le \mu_Y
$$
and
$W_\infty(\tilde\mu_X,\tilde\mu_Y) \le \ep $.
\end{enumerate}
In particular, $\mu_X$ and $\mu_Y$ are relative $(\ep,0)$-close
iff $W_\infty(\mu_X,\mu_Y) \le \ep $
\end{proposition}

For comparison of mm-spaces we have the following
corollary, which avoids explicit mentioning of metrics
on disjoint unions.

\begin{corollary}
\label{c:mm-wasserstein}
Let $X,Y$ be compact mm-spaces and $\ep,\de\ge 0$.
Then the following two conditions are equivalent:
\begin{enumerate}
 \item[(i)] 
$X$ and $Y$ are mm-relative $(\ep,\de)$-close
(see Definition \ref{d:mm-close}).
 \item[(ii)]
There exist measures $\tilde\mu_X$ on $X$ and $\tilde\mu_Y$ on $Y$
such that
\be\label{e:mm-wasserstein1}
e^{-\delta}\mu_X \le \tilde\mu_X \le \mu_X, 
\qquad 
e^{-\delta}\mu_Y \le \tilde\mu_Y \le \mu_Y
\ee
and a measure coupling $\ga$ between
$(X,\tilde\mu_X)$ and $(Y,\tilde\mu_Y)$
such that
\be\label{e:mm-wasserstein2}
 |d_X(x_1,x_2)-d_Y(y_1,y_2)| \le 2\ep
\ee
for all pairs $(x_1,y_1),(x_2,y_2)\in\supp(\ga)$.
\end{enumerate}
\end{corollary}

In particular, $(\ep,0)$-closeness of mm-spaces is
equivalent to $\ep$-closeness with respect to the
$L_\infty$ Gromov--Wasserstein distance.

The proof of Proposition \ref{p:prokhorov-vs-wasserstein}
and Corollary \ref{c:mm-wasserstein} occupies the rest of this section.
We prove the proposition
by means of discrete approximations.
We begin with a version of it for bipartite graphs.

Let $G=(V,E)$ be a bipartite graph with partite sets $M$ and $W$.
That is, the set $V$ of vertices is the union of disjoint sets $M$ and $W$
and each edge connects a vertex from $M$ to a vertex from $W$.
(Exercise: guess where the notations $M$ and $W$ came from.)
For a set $A\subset V$ we denote by $N_G(A)$ its graph neighborhood,
i.e., the set of vertices adjacent to at least one vertex from $A$.
A \textit{matching} in $G$ is a set of pairwise disjoint edges.

The classic Hall's Marriage Theorem \cite{Hall}
states the following. If for every set $A\subset M$ one has $|N_G(A)|\ge|A|$,
then there exists a matching that covers $M$
(that is, the set of endpoints of the matching contains $M$).
For discussion of Hall's Theorem and related topics see 
e.g.\ \cite[Ch.~7]{Ore}.
We need the following generalization of Hall's Theorem.

\begin{lemma}[Dulmage-Mendelsohn \cite{DM}]
\label{l:marriage}
Let $G=(V,E)$ be a bipartite graph with partite sets $M$ and $W$.
Let $M_0\subset M$ and $W_0\subset W$ be sets such that,
for every subset $A$ of either $M_0$ or $W_0$ one has $|N_G(A)|\ge |A|$.
Then $G$ contains a matching that covers $M_0\cup W_0$.
\end{lemma}

This lemma is proven as Theorem 1 in \cite{DM}.
It can also be seen as a combination of Hall's Theorem
and Ore's Mapping Theorem, see \cite[Theorem 7.4.1]{Ore}
or \cite[Theorem 2.3.1]{Mirsky}.

The next lemma is a ``continuous'' generalization of
Lemma \ref{l:marriage} where finite sets $M$ and $W$
are replaced by metric spaces $X$ and $Y$,
and a closed set $E\subset X\times Y$
plays the role of the set of edges of the graph.

\begin{lemma}
\label{l:Hall} 
Let $X$ and $Y$ be compact metric spaces.
Let $\mu_X^{}$, $\mu'_X$ be finite Borel measures on $X$
and $\mu_Y^{}$, $\mu'_Y$ finite Borel measures on $Y$
such that $\mu_X^{}\ge\mu'_X$ and $\mu_Y^{}\ge\mu'_Y$.

Let $E \subset X \times Y$ be a closed set.
Suppose that,
for any Borel sets $A\subset X$ and $B \subset Y$ one has
\be \label{1.23.12}
\mu_Y^{}(A^E) \geq \mu'_X(A),\quad \mu_X^{}(B^E) \geq \mu'_Y(B),
\ee
where
\begin{align*}
 A^E &= \{y \in Y: \text{there is $x\in A$ such that $(x,y)\in E$} \}, \\
 B^E &= \{x \in X: \text{there is $y\in B$ such that $(x,y)\in E$} \}.
\end{align*}
Then there exist measures $\tilde \mu_X$, $\tilde \mu_Y$ such that
\be \label{4.23.12}
\mu_X'\leq \tilde \mu_X^{} \leq \mu_X^{},      \quad \mu_Y'\leq \tilde \mu_Y^{} \leq \mu_Y^{},
\ee
and there is a measure coupling $\ga$ between $\tilde \mu_X$ and $\tilde \mu_Y$ 
such that $\supp(\ga) \subset E$.
\end{lemma}

\begin{proof}
First we prove the lemma in the special case when $X$ and $Y$ are finite sets.
By means of approximation we may assume that all values of the measures
$\mu_X^{},\mu_X',\mu_Y^{},\mu_Y'$ are rational numbers.
Multiplying by a common denominator we make them integers.
Then we derive the statement from Lemma \ref{l:marriage} as follows.

Split each point $x\in X$ into $\mu_X^{}(x)$ points of unit weight 
(keep in mind that $\mu_X^{}(x)\in\Z$).
Paint $\mu'_X(x)$ of these points in red and
the remaining $\mu_X^{}(x)-\mu'_X(x)$ points in green.
Similarly, split each point $y\in Y$ into $\mu_Y^{}(y)$ points
of which $\mu'_Y(y)$ are red and the rest are green.
Let $M$ and $W$ be the sets of points descending from points 
of $X$ and $Y$, resp.
Let $M_0$ and $W_0$ be the sets of red points
from $M$ and $W$, resp.

Now construct a bipartite graph $G$ with partite sets $M$ and $W$ as follows.
For $x\in X$ and $y\in Y$ such that $(x,y)\in E$, connect every
descendant of $x$ to every descendant of $y$ by an edge in $G$.
If $(x,y)\notin E$ then there are no edges between descendants of $x$ and~$y$.

The relation \eqref{1.23.12} implies that the graph $G$ satisfies
the assumptions of Lemma \ref{l:marriage}.
Therefore $G$ contains a matching $E_0$ covering $M_0\cup W_0$.
For each pair $(x,y)\in X\times Y$ 
define a point measure $\ga(x,y)$ equal to
the number of edges from $E_0$ connecting descendants of $x$ and~$y$.
Then $\ga$ is a desired coupling between some measures $\tilde\mu_X$
and $\tilde\mu_Y$ satisfying \eqref{4.23.12}.
Thus we are done with the discrete case.

Passing to the general case,
fix a sequence $\sigma_n\to 0$ of positive numbers.
For each $n$, divide $X$ and $Y$ into a finite number of Borel subsets $\Omega_X^i$, $\Omega_Y^j$ with
$\diam(\Omega_X^i)<\sigma_n$ and $\diam(\Omega_Y^j)<\sigma_n$.
Choose  points $x_i\in\Omega_X^i,\, y_j\in\Omega_Y^j$ 
and associate to them point measures
$\mu_{i,n}^{}=\mu_X(\Omega^i_X),\,  \mu'_{i,n}=\mu'_X(\Omega^i_X)$ and
$\mu_{j,n}^{}=\mu_Y(\Omega^j_Y),\,  \mu'_{j,n}=\mu'_Y(\Omega^j_Y)$.
This defines atomic measures $\mu_{X,n}^{},\mu'_{X,n}$ on $X$
and $\mu_{Y,n}^{},\mu'_{Y,n}$ on $Y$
and the relation \eqref{1.23.12} holds for these discrete measures 
with $E_n$ in place of $E$, where $E_n$ is the $2\sigma_n$-neighborhood of $E$ 
with respect to the product distance on $X \times Y$.

By the discrete case proven above,
there is a measure $\ga_n$ on $X\times Y$
whose marginals $\tilde\mu_{X,n}$ and $\tilde\mu_{Y,n}$
satisfy
\be \label{2.23.12}
\mu'_{X, n} \leq \tilde \mu_{X, n}^{}  \leq \mu_{X, n}^{},
\quad 
\mu'_{Y, n} \leq \tilde \mu_{Y, n}^{}  \leq \mu_{Y, n}^{}
\ee
and such that $\supp(\ga_n) \subset E_n$.
By the weak compactness of the space of measures
we may assume that the sequences $\tilde\mu_{X,n}^{}$, 
$\tilde\mu_{Y,n}^{}$ and $\ga_n$ weakly converge
to some measures $\tilde\mu_X^{}$, $\tilde\mu_Y^{}$ and~$\ga$, resp.
Then $\supp(\ga)\subset E$ and
$\ga$ is a measure coupling between $\tilde\mu_X$ and $\tilde\mu_Y$.
Also observe that the measures $\mu_{X,n}^{}$, $\mu'_{X,n}$,
$\mu_{Y,n}^{}$, $\mu'_{Y,n}$ weakly converge
to $\mu_X^{}$, $\mu'_X$, $\mu_Y^{}$, $\mu'_Y$, resp.
This and \eqref{2.23.12} imply that $\tilde\mu_X^{}$ and $\tilde\mu_Y^{}$
satisfy \eqref{4.23.12}.
\end{proof}

\begin{proof}[Proof of Proposition \ref{p:prokhorov-vs-wasserstein}]
Let $X=\supp(\mu_X)$ and $Y=\supp(\mu_Y)$.
The implication (i)$\Rightarrow$(ii)
follows from Lemma \ref{l:Hall} by substituting
$\mu'_X=e^{-\de}\mu_X$, $\mu'_Y=e^{-\de}\mu_Y$,
and 
$$
E = \{ (x,y)\in X\times Y :\, d(x,y)\le\ep \} .
$$
To prove the implication (ii)$\Rightarrow$(i),
let $\tilde\mu_X$ and $\tilde\mu_Y$ be as in
Proposition \ref{p:prokhorov-vs-wasserstein}(ii),
and let $\ga$ be a measure coupling
between $\tilde\mu_X$ and $\tilde\mu_Y$
realizing the $L^\infty$-Wasserstein distance.
Then $\supp(\ga)\subset E$.
This implies that, for every Borel set $A\subset X$,
$$
 \tilde\mu_X(A) = \ga(A\times Y)
 \le \ga(X\times(A^\ep\cap Y))
 = \tilde\mu_Y(A^\ep) ,
$$
where the inequality follows from the inclusion
$$
 (A\times Y)\cap E \subset X\times(A^\ep\cap Y) .
$$
Therefore
$$
 \mu_X(A) \le e^\de\,\tilde\mu_X(A)
 \le e^\de\,\tilde\mu_Y(A^\ep)
 \le e^\de\mu_Y(A^\ep) .
$$
Similarly $\mu_Y(B) \le e^\de\mu_X(B^\ep)$
for every Borel set $B\subset Y$.
Thus $\mu_X$ and $\mu_Y$ are 
relative $(\ep,\de)$-close.
\end{proof}

\begin{proof}[Proof of Corollary \ref{c:mm-wasserstein}]
(i)$\Rightarrow$(ii):
By definition, there exists a semi-metric $d$ on the disjoint union $Z=X\sqcup Y$ 
such that $\mu_X$ and $\mu_Y$, regarded as measures on~$Z$, are relative $(\ep,\de)$-close.
Proposition \ref{p:prokhorov-vs-wasserstein} implies that there exist
measures $\tilde\mu_X$ and $\tilde\mu_Y$ satisfying \eqref{e:mm-wasserstein1}
and a measure coupling $\ga$ between them such that
$d(x,y)\le\ep$ for all $(x,y)\subset\supp\ga$.
This property and the triangle inequality implies \eqref{e:mm-wasserstein2}.

(ii)$\Rightarrow$(i):
The proof is similar to that of \cite[Theorem 7.3.25]{BBI}.
Let $\ga$ be a measure coupling between $\tilde\mu_X$ and $\tilde\mu_Y$
such that \eqref{e:mm-wasserstein1} and \eqref{e:mm-wasserstein2} are satisfied.
Define a semi-metric $d$ on $X\sqcup Y$ by setting $d|_{X\times X}=d_X$,
$d|_{Y\times Y}=d_Y$, and
$$
 d(x,y) = \inf_{(x',y')\in\supp(\ga)} \left\{ d_X(x,x')+d_Y(y,y')+\ep \right\} .
$$
The triangle inequality for $d$ easily follows from \eqref{e:mm-wasserstein2},
thus $d$ is indeed a semi-metric.
The definition of $d$ implies that $d(x,y)=\ep$ if $(x,y)\in\supp(\ga)$.
Therefore $W_\infty(\tilde\mu_X,\tilde\mu_Y)\le\ep$ where $\tilde\mu_X$ and $\tilde\mu_Y$
are regarded as measures on~$Z$.
By Proposition \ref{p:prokhorov-vs-wasserstein} this implies that
$\mu_X$ and $\mu_Y$ are relative $(\ep,\de)$-close
and hence the mm-spaces $X$ and $Y$ are $(\ep,\de)$-close.
\end{proof}

\section{Stability of eigenvalues}
\label{sec:stability}

In this section we formulate and prove Theorem \ref{t:stability}
which is one of the main results of this paper.
Informally it says that if two mm-spaces are close then the lower parts of
spectra of their $\rho$-Laplacians are close.
First we introduce conditions on mm-spaces
needed in the theorem.

\begin{definition}[SLV condition] \label{d:SLV}
Let $X$ be a mm-space and $\Lambda,\rho,\ep>0$. We say that $X$
satisfies the \emph{spherical layer volume condition} with parameters $\Lambda$,
$\rho$, $\ep$, if for every $x\in \supp(\mu_X)$,
\be\label{e:SLV}
\frac{\mu(B_{\rho+\ep}(x) \setminus B_{\rho}(x))}{\mu(B_{\rho}(x))} 
\leq \Lambda \,\frac{\ep}{\rho}.
\ee
We abbreviate this condition as $SLV(\Lambda,\rho,\ep)$.
\end{definition}

\begin{definition}[BIV condition] \label{d:BIV}
Let $X$ be a mm-space, $0<\ep\le\rho/2$ and $\Lambda>0$. We say that $X$
satisfies the \emph{ball intersection volume condition} with parameters 
$\Lambda$, $\rho$, and $\ep$, if for all $x,y\in\supp(\mu_X)$ such that $d_X(x,y)\le\rho+\ep$,
$$
 \mu(B_{\rho}(x)\cap B_{\rho}(y)) \ge \Lambda^{-1} \mu(B_{\rho+\ep}(x)).
$$
We abbreviate this condition as $BIV(\Lambda,\rho,\ep)$.
\end{definition}

Note that the Bishop--Gromov inequality \eqref{e:bishop-gromov}
implies $SLV(\Lambda',\rho,\ep)$ for all $\rho\ge\ep>0$
with $\Lambda'$ depending on the parameter $\Lambda$ of \eqref{e:bishop-gromov}.
For $\ep=\rho$, the SLV condition \eqref{e:SLV} turns into a doubling condition:
$$
 \mu(B_{2\rho}(x)) \le 2\Lambda\mu(B_\rho(x)) .
$$
If $d$ is a length metric and this doubling condition
holds for all $\rho>0$, then $X$ satisfies
$BIV(\Lambda',\rho,\ep)$ for all $\rho>0$ and $\ep\le\rho/2$,
where $\Lambda'$ depends only on $\Lambda$.
This follows from the fact that the intersection
$B_\rho(x)\cap B_\rho(y)$ contains a
ball of radius $\frac{\rho-\ep}2$.

The next lemma shows that the conditions SLV and BIV
are in a sense stable with respect to
$(\ep,\de)$-closeness introduced in Section \ref{sec:wasserstein}.

\begin{lemma}
\label{l:conditions-stable}
Let $X$ and $Y$ be $(\ep,\de)$-close mm-spaces 
(see Definition \ref{d:mm-close})
where $0<\ep\le\rho/12$. 
Then:

1. If $X$ satisfies $SLV(\Lambda,\rho-2\ep,5\ep)$,
then $Y$ satisfies $SLV(6e^{2\de}\Lambda,\rho,\ep)$.

2. If $X$ satisfies $BIV(\Lambda,\rho-2\ep,5\ep)$,
then $Y$ satisfies $BIV(e^{2\de}\Lambda,\rho,\ep)$.
\end{lemma}

\begin{proof}
We may assume that $\mu_X$ and $\mu_Y$ have full support.
By definition, there is a metric $d$ on the disjoint union
$Z=X\sqcup Y$ such that $\mu_X$ and $\mu_Y$
are relative $(\ep,\de)$-close in $(Z,d)$.
Throughout this proof all balls, neighborhoods, etc, are considered
in the space $(Z,d)$.
Since the measures have full support,
the Hausdorff distance between $X$ and $Y$
is no greater than~$\ep$.
That is, for every $y\in Y$ there exists $x\in X$
such that $d(x,y)\le\ep$, and vice versa.

Let $y\in Y$. Take $x\in X$ such that $d(x,y)\le\ep$.
Recall that $A^\ep$ denotes the closed $\ep$-neighborhood of a set $A$.
The triangle inequality implies that
$$
  (B_{\rho-2\ep}(x))^\ep\subset B_{\rho}(y)
$$
and
$$
 (B_{\rho+\ep}(y)\setminus B_{\rho}(y))^\ep
 \subset B_{\rho+3\ep}(x)\setminus B_{\rho-2\ep}(x) .
$$
These inclusions and the relative $(\ep,\de)$-closeness 
of $\mu_X$ and $\mu_Y$ imply that
$$
 \mu_Y(B_{\rho}(y)) \ge e^{-\de} \mu_X(B_{\rho-2\ep}(x))
$$
and
$$
 \mu_Y (B_{\rho+\ep}(y)\setminus B_{\rho}(y))
 \le e^\de\mu_X(B_{\rho+3\ep}(x)\setminus B_{\rho-2\ep}(x)).
$$
Therefore
$$
\frac{\mu(B_{\rho+\ep}(y) \setminus B_{\rho}(y))}{\mu(B_{\rho}(y))} 
\le \frac{\mu(B_{\rho+3\ep}(x) \setminus B_{\rho-2\ep}(x))}{\mu(B_{\rho-2\ep}(x))} 
\le e^{2\de}\Lambda \,\frac{5\ep}{\rho-2\ep}
\le 6\Lambda\frac\ep\rho 
$$
and the first claim of the proposition follows.

To prove the second claim, consider points $y_1,y_2\in Y$
such that $d(y_1,y_2)\le\rho+\ep$. 
We have to prove that
$$
 Q := \frac{\mu_Y(B_{\rho}(y_1)\cap B_{\rho}(y_2))}{\mu_Y(B_{\rho+\ep}(y_1))}
 \ge (e^{2\de}\Lambda)^{-1} .
$$
Choose $x_1,x_2\in X$ such that
$d(x_1,y_1)\le\ep$ and $d(x_2,y_2)\le\ep$.
The triangle inequality implies that
$d(x_1,x_2)\le\rho+3\ep$,
$$
 (B_{\rho+\ep}(y_1))^\ep \subset B_{\rho+3\ep}(x_1)
$$
and
$$
 (B_{\rho-2\ep}(x_1)\cap B_{\rho-2\ep}(x_2))^\ep \subset
 B_{\rho}(y_1)\cap B_{\rho}(y_2) .
$$
Therefore, by relative $(\ep,\de)$-closeness of $\mu_X$ and $\mu_Y$,
$$
 \mu_Y(B_{\rho+\ep}(y_1)) \le e^\de \mu_X(B_{\rho+3\ep}(x_1))
$$
and
$$
 \mu_Y(B_{\rho}(y_1)\cap B_{\rho}(y_2))
 \ge \ep^{-\de}\mu_X(B_{\rho-2\ep}(x_1)\cap B_{\rho-2\ep}(x_2)) .
$$
Hence
$$
 Q \ge e^{-2\de} \, \frac{\mu_X(B_{\rho-2\ep}(x_1)\cap B_{\rho-2\ep}(x_2))}{\mu_X(B_{\rho+3\ep}(x_1))}
 \ge e^{-2\de}\Lambda^{-1}
$$
where the last inequality follows from the BIV condition for $X$.
This finishes the proof of Lemma \ref{l:conditions-stable}.
\end{proof}

Now we are in a position to state our main theorem.

\begin{theorem}\label{t:stability}
For every $\Lambda>0$ there exists $C=C(\Lambda)>0$
such that the following holds.
If $X$ and $Y$ are mm-spaces which are $(\ep,\de)$-close and
satisfy the conditions $SLV(\Lambda,\rho,2\ep)$ and $BIV(\Lambda,\rho,2\ep)$,
$0\le\ep\le\rho/4$, $\de\ge 0$,
then
\be\label{e:main-estimate}
  e^{-4\de}(1+C\ep/\rho)^{-1}\le\frac{\la_k(X,\rho)}{\la_k(Y,\rho)} \le e^{4\de}(1+C\ep/\rho)
\ee
for all $k$ such that $\la_k(X,\rho)<  e^{-4\de}(1+C\ep/\rho)^{-1}\rho^{-2}$.
\end{theorem}

The proof of Theorem \ref{t:stability} occupies the rest of this section.
First we prove the theorem for $\de=0$
(see Proposition \ref{p:ep-0-close}).
In this case Corollary \ref{c:mm-wasserstein} implies
that the mm-spaces $X$ and $Y$ in question admit
a measure coupling $\ga$ satisfying \eqref{e:mm-wasserstein2}.

To estimate the difference between eigenvalues 
of $\Dr_X$ and $\Dr_Y$,
we transform $X$ to $Y$ in three steps.
In the case when $X$ and $Y$ are discrete spaces 
these steps can be described as follows. 
First, we split each atom of $X$ into several points
and distribute the measure between them.
The distances between the descendants of each atom
is set to be zero, so we obtain a semi-metric-measure space.
Second, we ``transport'' the points to their destinations in $Y$.
The formal meaning of this is that we keep the point set
and the measure but change distances between points.
Finally, we glue together some points to obtain $Y$. 
The last step is inverse to the first one with $Y$ in place $X$.

After we provided this intuition in the discrete case,
let us proceed with a formal construction of ``splitting''.
It is slightly more cumbersome.

Let $\ga$ a measure coupling
between mm-spaces $X$ and $Y$.
Recall that $\ga$ is a measure on $X\times Y$ and
for every Borel set $A \subset X$,
\be \label{4.21.12}
\mu_X(A)= \ga(A \times Y).
\ee
Define a semi-metric $d_{X|X\times Y}$ on $X\times Y$ by
$$
d_{X|X\times Y}((x_1, y_1), (x_2, y_2))=d_X(x_1, x_2).
$$
The desired splitting of $X$ is the mm-space $X_\ga=(X\times Y,d_{X|X\times Y},\ga)$.

We do not use the (non-Hausdorff) topology arising from the semi-metric $d_{X|X\times Y}$.
We equip $X\times Y$ with the standard product Borel $\sigma$-algebra.

An interested reader may check that the arguments below also apply
if one replaces the semi-metric $d_{X|X\times Y}$ by a genuine metric $d$
defined by 
$$
d((x_1, y_1), (x_2, y_2))=\max\{d_X(x_1, x_2),c\rho^{-2}d_Y(y_1,y_2)\}
$$
where $c$ is a sufficiently small constant, $0<c<1/\diam(Y)$.

Applying Definition \ref{d:rho-laplacian} to $X_\ga$ 
we define the associated $\rho$-Laplacian $\Dr_{X_\ga}$.
Even though $X_\ga$ is almost the same space as $X$,
the spectrum of $\Dr_{X_\ga}$ may slightly differ from
that of~$\Dr_X$.
We compare the two spectra in the following lemma:

\begin{lemma}\label{l:splitting}
Let $X$ and $X_\ga$ be as above. Then
$$
\spec(\Dr_{X_\ga})\subset\spec(\Dr_X)\cup\{\rho^{-2}\} .
$$
Furthermore, every eigenvalue smaller than $\rho^{-2}$
has the same multiplicity in the two spectra.
\end{lemma}

\begin{proof}
Consider a subspace $L\subset L^2(X\times Y,\ga)$
given by $L=\pi^*_X(L^2(X))$ where $\pi_X\co X \times Y \to X$ is the coordinate projection.
In other words, $L$ consists of functions 
which are constant on every fiber $\{x\}\times Y$, $x\in X$.
Due to \eqref{4.21.12}, $\pi_X^*$ is a Hilbert space isomorphism
between $L^2(X)$ and $L$.
We decompose $L^2(X\times Y,\ga)$
into a direct sum $L\oplus L^\perp$.
Loosely speaking, $L^\perp$ consists of functions 
which are orthogonal to constants in every fiber.
More precisely, if $u\in L^\perp$ then
\be\label{e:splitting1}
 \int_{A\times Y} u\,d\ga = 0
\ee
for every Borel set $A\subset X$.

The statement of the lemma is a consequence of
the following three facts:
\begin{enumerate}
 \item[(1)]  $L$ and $L^\perp$ are invariant under $\Dr_{X_\ga}$;
 \item[(2)] $\pi_X^*$ provides an equivalence between
$\Dr_X$ and $\Dr_{X_\ga}|_L$;
 \item[(3)] for every $u\in L^\perp$ we have $\Dr_{X_\ga}u=\rho^{-2}u$.
\end{enumerate}

To prove these facts, observe that a $\rho$-ball $B_\rho^{X_\ga}(x,y)$
of the semi-metric $d_{X|X\times Y}$ is of the form
$$
 B_\rho^{X_\ga}(x,y) = B_\rho^X(x) \times Y .
$$
Hence for a function $u=\pi_X^*(v)\in L$, where $v\in L^2(X)$, we have
\begin{align*}
 \Dr_{X_\ga}u(x,y) 
 &= \frac{\rho^{-2}}{\ga(B^X_\rho(x)\times Y)}\int_{B^X_\rho(x)\times Y} [u(x,y)-u(x_1,y_1)] \,d\ga(x_1,y_1) \\
 &= \frac{\rho^{-2}}{\mu_X(B^X_\rho(x))}\int_{B^X_\rho(x)} [v(x)-v(x_1)] \,d\mu_X(x_1)
 = \Dr_X v(x)
\end{align*}
where the second identity follows from \eqref{4.21.12}.
Thus $\Dr_{X_\ga}(\pi_X^*(v))=\pi_X^*(\Dr_X v)$,
proving (2) and the first part of (1).

For every $u\in L^\perp$ we have
\begin{align*}
 \Dr_{X_\ga}u(x,y)
 &= \frac{\rho^{-2}}{\ga(B^X_\rho(x)\times Y)}\int_{B^X_\rho(x)\times Y} [u(x,y)-u(x_1,y_1)] \,d\ga(x_1,y_1) \\
 &= \rho^{-2} u(x,y) -\frac{\rho^{-2}}{\ga(B^X_\rho(x)\times Y)}\int_{B^X_\rho(x)\times Y} u(x_1,y_1) \,d\ga(x_1,y_1) \\
 &= \rho^{-2} u(x,y)
\end{align*}
where the last identity follows from \eqref{e:splitting1}.
This proves (3) and the second part of (1).
\end{proof}

The next lemma serves the step where we handle
the difference between distances.

\begin{lemma}\label{l:distance-change}
Let $X_1=(X,d_1,\mu)$, $X_2=(X,d_2,\mu)$
be two mm-spaces with the same
point set $X$ and measure $\mu$. 
Let $\Lambda\ge 1$, $0<\ep\le\rho/2$,
and assume that $X_1$, $X_2$ satisfy the conditions
$SLV(\Lambda,\rho,\ep)$ and $BIV(\Lambda,\rho,\ep)$.
Also assume that
\be\label{e:dist-error}
 |d_1(x,y)-d_2(x,y)| \le \ep
\ee
for all $x,y\in\supp(\mu)$.
Then, for every $k\in\N$,
\be\label{e:la-change}
 (1+C\ep/\rho)^{-1}\le\frac{\la_k(X_2,\rho)}{\la_k(X_1,\rho)} \le 1+C\ep/\rho
\ee
where $C$ is a constant depending only on $\Lambda$.
\end{lemma}

\begin{proof}
We may assume that \eqref{e:dist-error} holds for all $x,y\in X$,
otherwise just replace $X$ by $\supp(\mu)$.
We estimate the eigenvalues by means of the min-max formula \eqref{e:minmax}.
For $i=1,2$, let $B^i_\rho(x)$ denote the $\rho$-ball of $d_i$ centered at $x\in X$,
$\|\cdot\|_i=\|\cdot\|_{X_i^\rho}$  (see \eqref{e:L2norm}),
and $D_i=D^\rho_{X_i}$ (see \eqref{e:dirichlet}). 
The only difference as we pass from $X_1$ to $X_2$
is that the balls $B^i_\rho(x)$ are different.

The assumption \eqref{e:dist-error} implies that $B^1_\rho(x)\subset B^2_{\rho+\ep}(x)$
for every $x\in X$.
This and the condition $SLV(\Lambda,\rho,\ep)$ for $X_2$ imply that
$$
 \frac{\mu(B^1_\rho(x))}{\mu(B^2_\rho(x))}
 \le 1+\frac{\mu(B^2_{\rho+\ep}(x)\setminus B^2_\rho(x))}{\mu(B^2_\rho(x))}
 \le 1+\Lambda\ep/\rho .
$$
This and \eqref{e:L2norm} imply that
\be\label{e:dist-change0}
 \|u\|^2_1 \le (1+\Lambda\ep/\rho) \|u\|^2_2
\ee
for every $u\in L^2(X)$.

For the Dirichlet forms we have
\begin{align*}
 D_2(u)  &= \iint_{d_2(x, y) <\rho} |u(x)-u(y)|^2 \, d\mu(x) d\mu(y) \\
 &\leq \iint_{d_1(x, y) <\rho+\ep} |u(x)-u(y)|^2 \, d\mu(x) d\mu(y) \\
 &= D_1(u) +\iint_{L} |u(x)-u(y)|^2 \, d\mu(x) d\mu(y),
\end{align*}
where 
$$
L=\{(x,y)\in X\times X : \rho \le d_1(x, y) <\rho+\ep \} .
$$
Hence
\be\label{e:dist-change1}
 D_2(u)-D_1(u) \le \iint_{L} |u(x)-u(y)|^2 \, d\mu(x) d\mu(y) .
\ee

Let us estimate the right-hand side of \eqref{e:dist-change1}.
For every $(x,y)\in L$ consider the set
$U(x,y)=B_\rho^1(x) \cap B_\rho^1(y)$.
Recall that
\be\label{e:dist-change2}
 \mu(U(x,y)) \ge \Lambda^{-1} \max\{\mu(B^1_\rho(x)),\mu(B^1_\rho(y))\}
\ee
by the condition $BIV(\Lambda,\rho,\ep)$ for $X_1$.
For every $z \in U(x, y)$ we have
$$
|u(x)-u(y)|^2 \leq 2 (|u(x)-u(z)|^2+|u(z)-u(y)|^2).
$$
Integrating this inequality and taking into account \eqref{e:dist-change2} yields that
\begin{align*}
|u(x)-u(y)|^2 &\le \frac{2}{\mu_1(U(x,y))} \int_{U(x, y)} (|u(x)-u(z)|^2+|u(z)-u(y)|^2) \, d\mu_1(z) \\
&\le 2\Lambda \bigl(Q(x)+Q(y)\bigr)
\end{align*}
where
$$
 Q(x) = \frac1 {\mu(B^1_\rho(x))} \int_{B^1_\rho(x)} |u(x)-u(z)| ^2\, d\mu_1(z) .
$$
This and \eqref{e:dist-change1} imply that
\begin{align*}
D_2(u)-D_1(u) 
&\le 2\Lambda \iint_L (Q(x)+Q(y))\,d\mu(x)d\mu(y) \\
&= 4\Lambda \iint_L Q(x)\,d\mu(x)d\mu(y) \\
&= 4\Lambda \int_X \mu(B^1_{\rho+\ep}(x)\setminus B^1_\rho(x))\, Q(x) \,d\mu(x) \\
&\le \frac{4\Lambda^2\ep}{\rho} \int_X \mu(B^1_\rho(x))\, Q(x) \,d\mu(x) \\
&=\frac{4\Lambda^2\ep}{\rho} D_1(u)
\end{align*}
where the second inequality follows from 
the condition $SLV(\Lambda,\rho,\ep)$ for $X_1$.
Thus
\be\label{e:dist-change-D}
 D_2(u) \le (1+4\Lambda^2\ep/\rho) D_1(u) .
\ee
This and \eqref{e:dist-change0} imply that
$$
  \frac{D_2(u)}{\|u\|^2_2} \le (1+\Lambda\ep/\rho)(1+4\Lambda^2\ep/\rho) \frac{D_1(u)}{\|u\|^2_1}
$$
for every $u\in L^2(X)\setminus\{0\}$.
By the min-max formula \eqref{e:minmax} this implies the second inequality
in \eqref{e:la-change} with $C=\Lambda+4\Lambda^2+4\Lambda^3$. 
Then the first inequality in  \eqref{e:la-change}
follows by swapping $X_1$ and $X_2$.
\end{proof}

The following proposition deals with the case of $\de=0$
of Theorem \ref{t:stability}.

\begin{proposition}
\label{p:ep-0-close}
For every $\Lambda>0$ there exists $C=C(\Lambda)>0$
such that the following holds.
Let $0<\ep\le\rho/4$ and 
let $X$, $Y$ be mm-spaces that are $(\ep,0)$-close
and satisfy the conditions 
$SLV(\Lambda,\rho,2\ep)$ and $BIV(\Lambda,\rho,2\ep)$.
Then
$$
  (1+C\ep/\rho)^{-1} \le \frac{\la_k(X,\rho)}{\la_k(Y,\rho)} \le 1+C\ep/\rho
$$
for every $k\in\N$ such that $\la_k(X,\rho)<(1+C\ep/\rho)^{-1}\rho^{-2}$.
\end{proposition}

\begin{proof}
By Corollary \ref{c:mm-wasserstein}, there exists a measure coupling
$\ga$ between $X$ and~$Y$ satisfying \eqref{e:mm-wasserstein2}.
With this coupling, we construct mm-spaces 
$$
X_\ga=(X\times Y,d_{X|X\times Y},\ga)
\quad\text{and}\quad
Y_\ga=(X\times Y,d_{Y|X\times Y},\ga)
$$
as explained in the text before Lemma \ref{l:splitting}.
Then Lemma \ref{l:splitting} implies that 
$\la_k(X_\ga,\rho)=\la_k(X,\rho)$
provided that $\la_k(X,\rho)<\rho^{-2}$.

The spaces $X_\ga$ and $Y_\ga$ inherit the conditions $SLV(\Lambda,\rho,2\ep)$
and $BIV(\Lambda,\rho,2\ep)$ from $X$ and $Y$.
Due to \eqref{e:mm-wasserstein2}, $X_\ga$ and $Y_\ga$
satisfy the assumptions of Lemma \ref{l:distance-change}
with $2\ep$ in place of $\ep$.
Hence
$$
  (1+C\ep/\rho)^{-1} \le \frac{\la_k(X_\ga,\rho)}{\la_k(Y_\ga,\rho)} \le 1+C\ep/\rho 
$$
where $C$ is a constant depending only on $\Lambda$.
If $\la_k(X_\ga,\rho)<(1+C\ep/\rho)^{-1}\rho^{-2}$, this implies that $\la_k(Y_\ga,\rho)<\rho^{-2}$
and therefore $\la_k(Y_\ga,\rho)=\la_k(Y,\rho)$ by Lemma~\ref{l:splitting}.
The proposition follows.
\end{proof}

\begin{proof}[\bf Proof of Theorem \ref{t:stability}]
Let $X$, $Y$ be as in Theorem \ref{t:stability}.
By Corollary \ref{c:mm-wasserstein},
there exist measures $\tilde\mu_X$ and $\tilde\mu_Y$
satisfying \eqref{e:mm-wasserstein1} 
and such that the mm-spaces $\tilde X=(X,d_X,\tilde\mu_X)$ and $\tilde Y=(Y,d_Y,\tilde\mu_Y)$
are $(\ep,0)$-close.
By \eqref{e:la-change-mu} we have
\be\label{e:tildeX}
 e^{-2\de} \le \frac{\la_k(\tilde X,\rho)}{\la_k(X,\rho)} \le e^{2\de}
\ee
and
\be\label{e:tildeY}
 e^{-2\de} \le \frac{\la_k(\tilde Y,\rho)}{\la_k(Y,\rho)} \le e^{2\de} .
\ee

Now we estimate the ratio $\la_k(\tilde X,\rho)/\la_k(\tilde Y,\rho)$.
Due to \eqref{e:mm-wasserstein1}, 
$\tilde X$ and $\tilde Y$
satisfy the conditions $SLV(\Lambda',\rho,2\ep)$ and $BIV(\Lambda',\rho,2\ep)$
with $\Lambda'=e^\de\Lambda$.
By Proposition \ref{p:ep-0-close} applied to $\tilde X$ and $\tilde Y$
we have
\be\label{e:tilderatio}
  (1+C\ep/\rho)^{-1} \le \frac{\la_k(\tilde X,\rho)}{\la_k(\tilde Y,\rho)} \le 1+C\ep/\rho 
\ee
provided that
$
 \la_k(\tilde X,\rho) < (1+C\ep/\rho)^{-1}\rho^{-2} 
$. 
Here $C$ is a constant depending only on $\Lambda$.
The desired estimate \eqref{e:main-estimate}
follows from \eqref{e:tilderatio}, \eqref{e:tildeX} and \eqref{e:tildeY}.
\end{proof}

\begin{proof}[\bf Proof of Theorem \ref{t:second}]
Let $X=(X,d,\mu)$ and $X_n=(X_n,d_n,\mu_n)$ be as in Theorem \ref{t:second}.
The Bishop--Gromov condition \eqref{e:bishop-gromov}
implies that $\mu$ has full support.
By Proposition \ref{p:fukaya} it follows that
$X_n$ is $(\ep_n,\de_n)$-close to $X$ where $\ep_n,\de_n\to 0$.

Fix $\rho>0$ and assume that $\ep_n<\rho/24$.
As explained after Definition \ref{d:BIV},
the assumption that $d$ is a length metric and \eqref{e:bishop-gromov}
imply that $X$ satisfies $SLV(\Lambda',r,\ep)$
and $BIV(\Lambda',r,\ep)$ for all $r>0$ and $\ep\le r/2$,
where $\Lambda'$ depends only on $\Lambda$.
By Lemma \ref{l:conditions-stable} it follows that
$X_n$ satisfies $SLV(\Lambda'',\rho,2\ep_n)$
and $BIV(\Lambda'',\rho,2\ep_n)$ 
for some $\Lambda''$ depending only on $\Lambda$.
Now Theorem \ref{t:stability} implies that, for some $C=C(\Lambda)$,
$$
  e^{-4\de_n}(1+C\ep_n/\rho)^{-1}
  \le \frac{\la_k(X_n,\rho)}{\la_k(X,\rho)} 
  \le e^{4\de_n}(1+C\ep_n/\rho)
$$
for all $n,k$ such that 
$\la_k(X,\rho)<e^{-4\de_n}(1+C\ep_n/\rho)^{-1}\rho^{-2}$.
Thus $\la_k(X_n,\rho)\to\la_k(X,\rho)$ as $n\to\infty$.
\end{proof}

\section{Transport of  $\rho$-Laplacians}
\label{sec:TXY}

In this section we further analyze the structures
appeared in the proof of Theorem \ref{t:stability}.
Our goal is to construct a map $T_{XY}\co L^2(X)\to L^2(Y)$
which shows ``almost equivalence'' of
$\rho$-Laplacians $\Dr_X$ and $\Dr_Y$.
See Proposition \ref{p:TXY} for a precise formulation.

Let $X$, $Y$ be as in Theorem \ref{t:stability}.
As in the proof of Theorem \ref{t:stability},
let $\ga$ be a measure coupling 
provided by Corollary \ref{c:mm-wasserstein} and
$\tilde\mu_X$, $\tilde\mu_Y$ marginals of~$\ga$.
The coordinate projection $\pi_X\co X\times Y\to X$
determines two maps
$$
I_X\co L^2(X,\tilde\mu_X)\to L^2(X\times Y,\ga)
$$
and 
$$
P_X\co L^2(X\times Y,\ga)\to L^2(X,\tilde\mu_X)
$$
which are dual to each other.
Namely $I_X=\pi_X^*$ is a map given by
$$
 (I_Xu)(x,y) = u(x), \qquad u\in L^2(X),\ x\in X,\ y\in Y .
$$
Note that $I_X$ is an isometric embedding of $L^2(X,\tilde\mu_X)$
to $L^2(X\times Y,\ga)$. 
Let $L_X\subset L^2(X\times Y,\ga)$ be the image of $I_X$.
Then $P_X$ is the composition of the orthogonal projection
onto $L_X$ and the map $I_X^{-1}\co L_X\to L^2(X)$.

Loosely speaking, $P_X$ sends each function on $X\times Y$
to the family of its average values over the fibers $\{x\}\times Y$, $x\in X$.
More precisely, by disintegration theorem (see \cite{Pachl} or  
\cite[Theorem 452I]{Fremlin}), for a.e.\ $x\in X$, there is a measure $\nu_x$ on $Y$
such that
$$
\gamma(A \times B)= \int_A \nu_x(B) d\tilde \mu_X,\quad A \subset X, B \subset Y.
$$
Then, for a.e.\ $x\in X$,
\be\label{e:PX}
 (P^{}_X\varphi)(x) = 
 \int_{ Y} \phi(x, y)\, d\nu_x,
 \qquad x\in X.
\ee
 Then, since
$\tilde\mu_X$ is a marginal measure of $\ga$, \eqref{e:PX} implies that $P_X \circ I_X=id_X$.

Similarly one defines maps $I_Y$, $P_Y$ and a subspace $L_Y$.
We introduce a map $T_{XY}\co L^2(X)\to L^2(Y)$
by $T_{XY}=P_Y\circ I_X$.
By \eqref{e:PX}, for $u \in L^2(X)$,
$$
 (T^{}_{XY}u)(y) = \int_{X} u(x)  \, d\nu_y,
$$
where $\nu_y$ is defined similarly to $\nu_x$ and the 
integral in the right-hand side exists for a.e.~$y\in Y$.
The main result of this section is the following proposition.

\begin{proposition}
\label{p:TXY}
For every $\Lambda>0$ there exists $C=C(\Lambda)>0$ such that the following holds.
Let $X$, $Y$, $\rho$, $\ep$, $\Lambda$ be as in Theorem \ref{t:stability}.
Then for every $u\in L^2(X)$ the map $T_{XY}$ defined above satisfies
\be
\label{e:TXY01}
 A^{-1}\|u\|^2_{X^\rho} - A\rho^2 D^\rho_X(u)
 \le\|T^{}_{XY}u\|^2_{Y^\rho} \le A\|u\|^2_{X^\rho} ,
\ee
\be\label{e:TXY02}
 D^\rho_Y(T^{}_{XY}u) \le A D^\rho_X(u) ,
\ee
and
\be\label{e:TXY03}
 \| T_{YX}(T_{XY}u) - u \|^2_{X^\rho} \le A\rho^2 D^\rho_X(u) ,
\ee
where
$$
 A = e^{\de}(1+C\ep/\rho) .
$$
\end{proposition}

We are interested in the situation when $\de$ and $\ep/\rho$ are small.
Then $A$ is close to~1 and the cumbersome formulas
\eqref{e:TXY01} and \eqref{e:TXY03} can be informally
interpreted in the following way.
At not too high energy levels
(that is, if the Dirichlet energy of a unit vector $u$
is substantially smaller than $\rho^{-2}$),
the operator $T_{XY}$ almost preserves the norm and the inner product by \eqref{e:TXY01}
and $T_{YX}\circ T_{XY}$ is close to identity by \eqref{e:TXY03}.

\begin{remark}
Let us note that a construction of a similar type 
has been introduced by Sturm \cite{Sturm06-I}, see in particular
formula (4.36) in \cite[Sec.~4.5]{Sturm06-I}. 
There the author defines a map $Q'$ which is essentially the same as our $T_{XY}$,
in different notation.
In \cite{Sturm06-I}, Sturm was mostly after Ricci curvature bounds. 
Later Gigli \cite {Gi} used similar 
techniques with other problems in mind.
We are after quite different goals 
and therefore need different estimates and applications of the construction.
\end{remark}

\begin{proof}[Proof of Proposition \ref{p:TXY}]
Let $\ga$, $\tilde\mu_X$, $\tilde\mu_Y$ be as above.
Consider mm-spaces $\tilde X=(X,d_X,\tilde\mu_X)$ and
$\tilde Y=(Y,d_Y,\tilde\mu_Y)$,
the corresponding $\rho$-Laplacians,
norms $\|\cdot\|_{\tilde X^\rho}$ and $\|\cdot\|_{\tilde Y^\rho}$,
and Dirichlet forms $D^\rho_{\tilde X}$ and $D^\rho_{\tilde Y}$
(see \eqref{e:L2norm} and \eqref{e:dirichlet}).

As in Section \ref{sec:stability} we equip $X\times Y$ with
two semi-distances $d_{X|X\times Y}$ and $d_{Y|X\times Y}$
and denote by $\tilde X_\ga$ and $\tilde Y_\ga$ the corresponding mm-spaces
(see the proof of Proposition \ref{p:ep-0-close}).
These mm-spaces determine $\rho$-Laplacians
$\Dr_{\tilde X_\ga}$ and $\Dr_{\tilde Y_\ga}$,
scalar products
$\langle,\rangle_{\tilde X_\ga^\rho}$ and $\langle,\rangle_{\tilde Y_\ga^\rho}$,
and Dirichlet forms $D^\rho_{\tilde X_\ga}$ and $D^\rho_{\tilde Y_\ga}$
(see \eqref{e:L2rho} and \eqref{e:dirichlet}).
The structures introduced above 
satisfy the following properties
(see the proof of Lemma \ref{l:splitting}):
\begin{itemize}
\item 
$L_X$ and $L_X^\perp$ are orthogonal 
with respect to $\langle\cdot,\cdot\rangle_{\tilde X_\ga^\rho}$;
\item 
$I_X$ is an isometric embedding with respect 
to norms $\|\cdot\|_{\tilde X^\rho}$ and $\|\cdot\|_{\tilde X_\ga^\rho}$;
\item 
$I_X$ preserves the Dirichlet form $D^\rho_{\tilde X}$,
that is $D^\rho_{\tilde X_\ga}(I_Xu)=D^\rho_{\tilde X}(u)$
for all $u\in L^2(X)$;
\item
$L_X$ and $L_X^\perp$ are invariant under 
$\Dr_{\tilde X_\ga}$
and hence they are orthogonal with respect to $D^\rho_{\tilde X_\ga}$.
\end{itemize}
Recall that $P_X$ is the composition of the orthogonal projection
to $L_X$ and the map $I_X^{-1}$. Hence $P_X$ does not increase the norms and Dirichlet forms.
Similar properties hold for $Y$ in place of $X$.

As in the proof of Lemma \ref{l:distance-change},
for every $v\in L^2(X\times Y,\ga)$
we have
(see \eqref{e:dist-change0} and \eqref{e:dist-change-D})
\be\label{e:TXY-norm}
  A_1^{-1}\le \|v\|^2_{\tilde X^\rho} \big/ \|v\|^2_{\tilde Y^\rho}
  \le A_1
\ee
and
\be\label{e:TXY-D}
  A_2^{-1}\le D^\rho_{\tilde X}(v) \big/ D^\rho_{\tilde Y}(v)
  \le A_2
\ee
where $A_1=1+\Lambda\ep/\rho$ and $A_2=1+4\Lambda^2\ep/\rho$.

Let $u\in L^2(X)$ and $v=I_X(u)$.
Then
\be
\label{e:TXY1}
 \|T^{}_{XY}u\|^2_{\tilde Y^\rho}
 = \|P_Yv\|^2_{\tilde Y^\rho}
 \le \|v\|^2_{\tilde Y_\ga^\rho}
 \le A_1\|v\|^2_{\tilde X_\ga^\rho} 
 =A_1\|u\|^2_{\tilde X^\rho} .
\ee
Now we estimate $\|T_{XY}u\|^2_{Y^\rho}$ from below.
Decompose $v$ as $v=v_1+v_2$ where $v_1\in L_Y$ and $v_2\in L_Y^\perp$.
As shown in the proof of Lemma \ref{l:splitting},
the $\rho$-Laplacian $\Dr_{\tilde Y_\ga}$ 
acts on $L_Y^\perp$ by multiplication by $\rho^{-2}$.
Hence
\be\label{e:TXY4}
 D^\rho_{\tilde Y_\ga}(v) 
 = D^\rho_{\tilde Y_\ga}(v_1) + D^\rho_{\tilde Y_\ga}(v_2)
 \ge D^\rho_{\tilde Y_\ga}(v_2)
 = \rho^{-2}\|v_2\|^2_{\tilde Y_\ga^\rho}
\ee
and therefore
$$
 \|v_1\|^2_{\tilde Y_\ga^\rho} 
 =\|v\|^2_{\tilde Y_\ga^\rho} - \|v_2\|^2_{\tilde Y_\ga^\rho}
 \ge \|v\|^2_{\tilde Y_\ga^\rho} - \rho^2 D^\rho_{\tilde Y_\ga}(v) .
$$
Thus
\begin{multline}
\label{e:TXY2}
\|T^{}_{XY}u\|^2_{\tilde Y^\rho}
= \|v_1\|^2_{\tilde Y_\ga^\rho}
\ge \|v\|^2_{\tilde Y_\ga^\rho} - \rho^2 D^\rho_{\tilde Y_\ga}(v) \\
\ge A_1^{-1} \|v\|^2_{\tilde X_\ga^\rho}
- A^{}_2\rho^2 D^\rho_{\tilde X_\ga}(v) 
= A_1^{-1} \|u\|^2_{\tilde X^\rho}
- A^{}_2\rho^2 D^\rho_{\tilde X}(u)
\end{multline}
by \eqref{e:TXY-norm} and \eqref{e:TXY-D}.
Now \eqref{e:TXY01} follows from \eqref{e:TXY1}, \eqref{e:TXY2}
and the bounds \eqref{e:mm-wasserstein1} 
for $\tilde\mu_X$ and $\tilde\mu_Y$.

To estimate the Dirichlet form of $T_{XY}u$,
observe that
$$
 D^\rho_{\tilde Y}(T^{}_{XY}u) = D^\rho_{\tilde Y_\ga}(v_1)
 \le D^\rho_{\tilde Y_\ga}(v)
$$
and
\be\label{e:TXY5}
 D^\rho_{\tilde Y_\ga}(v)
 \le A_2 D^\rho_{\tilde X_\ga}(v)
 = A_2 D^\rho_{\tilde X}(u)
\ee
by \eqref{e:TXY-D}. These estimates and \eqref{e:mm-wasserstein1}
imply \eqref{e:TXY02}.

To prove \eqref{e:TXY03}, observe that $I_Y(T_{XY}u)=v_1$
and therefore
$$
T_{YX}(T_{XY}u) - u = P_X(v_1)-u = P_X(v_1-v) = P_X(v_2) .
$$
Further,
$$
 \|P_X(v_2)\|^2_{\tilde X^\rho} \le \|v_2\|^2_{\tilde X_\ga^\rho}
 \le A_1\|v_2\|^2_{\tilde Y_\ga^\rho}
 \le A_1\rho^2 D^\rho_{\tilde Y_\ga}(v)
 \le A_1A_2 \rho^2 D^\rho_{\tilde X}(u)
$$
by \eqref{e:TXY-norm}, \eqref{e:TXY4}, and \eqref{e:TXY5}.
Thus
$$
 \|T_{YX}(T_{XY}u) - u\|^2_{\tilde X^\rho} \le A_1A_2 \rho^2 D^\rho_{\tilde X}(u) .
$$
This and \eqref{e:mm-wasserstein1} imply \eqref{e:TXY03}.
\end{proof}

After we obtained estimates on closeness of eigenvalues
in Theorem \ref{t:stability} we certainly would like to
show that corresponding eigenspaces are also close.

The most naive formulation definitely fails.
If we have an eigenvalue of multiplicity 2
then there is a two-dimensional eigenspace.
Then a small perturbation would generically
result in splitting the eigenspace into two
orthogonal one-dimensional eigenspaces.
An original eigenvector may fail to be close
to either of the new eigenspaces.
It is still close to a linear combination
of new eigenvectors.

In our case we have a similar situation.
Let $u$ be an eigenvector of $\Dr_X$ with eigenvalue $\la$
which is substantially smaller that $\rho^{-2}$.
Then $T_{XY}(u)$ is close to a linear combination
of eigenvectors of $\Dr_Y$ with eigenvalues
close to~$\la$.
We don't give a precise formulation of the statement.
It is a direct reformulation of Theorem 3 in \cite{BIK14}.
The proof is an application of Proposition~\ref{p:TXY}
and some straightforward linear algebra.

\section{Weyl-type estimates}
\label{sec:weyl}

In this section we prove Theorems \ref{t:ess-spectrum}
and \ref{t:many-eigenvalues}.
Theorem \ref{t:ess-spectrum} 
gives us a Weyl-type upper bound on
the number of eigenvalues in a lower part of the spectrum of $\Dr_X$.
Theorem \ref{t:many-eigenvalues} provides a similar lower bound.

To formulate the theorems we need notation for packing numbers.
For a compact metric space $X$ and $r>0$ we denote
by $N_X(r)$ the maximum 
number of points in an $r$-separated set in $X$.
Recall that a set $Y\subset X$ is \emph{$r$-separated}
if $d_X(y_1,y_2)\ge r$ for all $y_1,y_2\in Y$.

For $R>0$, we denote by $\#^\rho_X(R)$
the number of eigenvalues of $\Dr_X$ in the interval $[0,R]$,
counted with multiplicities. Equivalently,
$$
\#_X^\rho(R) = \sup \{ k\in\N : \la_k(X,\rho)\le R \} .
$$
Note that $\#_X^\rho(R)=\infty$ if $R\ge\la_\infty(X,\rho)$.

\begin{theorem}
\label{t:ess-spectrum}
For every $\Lambda\ge 1$ there exists $c=c(\Lambda)>0$
such that the following holds.
Let $X$ be a mm-space satisfying the condition $BIV(\Lambda,\frac56\rho,\frac5{12}\rho)$
and the following restricted doubling condition:
$$
 \mu_X(B_{5\rho/3}) \le \Lambda \mu_X (B_{5\rho/6})
$$
for all $x\in\supp(\mu_X)$.
Then 
$$
 \#_X^\rho(c\rho^{-2}) \le N_X(\rho/24) . 
$$
\end{theorem}

If $X$ is a Riemannian manifold then $N_X(r) \sim C_n\mu(X)r^{-n}$
as $r\to 0$, 
where $n$ is the dimension of $X$.
In this case the conclusion of Theorem \ref{t:ess-spectrum}
can be restated as follows:
for $R=c\rho^{-2}$, we have 
$$
\#_X^\rho(R) \le C(n,\Lambda) \mu(X) R^{n/2} .
$$
The reader is invited to compare the right-hand side
of this formula with the classic Weyl's asymptotics
for the Beltrami--Laplace spectrum.

\begin{proof}[\bf Proof of Theorem \ref{t:ess-spectrum}]
Let $X$ be a mm-space satisfying the assumptions
of the theorem.
Fix $\ep=\rho/24$.
Let $Y$ be a maximal
$\ep$-separated set in~$X$
and $N=N_X(\ep)$ the cardinality of $Y$.
Then $Y$ is an $\ep$-net in $X$.
Equip $Y$ with a measure $\mu_Y$ as in Example \ref{x:discretization}
so that the resulting mm-space is $(\ep,0)$-close to $X$.

By the assumptions of the theorem,
$X$ satisfies the conditions $SLV(\Lambda,\rho-4\ep,10\ep)$
and $BIV(\Lambda,\rho-4\ep,10\ep)$.
By Lemma \ref{l:conditions-stable}
it follows that
$Y$ satisfies $SLV(6\Lambda,\rho,2\ep)$ and $BIV(\Lambda,\rho,2\ep)$.
Therefore Theorem~\ref{t:stability} applies to $X$ and $Y$ with $6\Lambda$
in place of $\Lambda$. 
By Theorem~\ref{t:stability}, for every $k\ge 1$ at least
one of the following holds:
either
$$
 \la_k(X,\rho) > C^{-1}\rho^{-2}
$$
or
$$
  C^{-1}\le\frac{\la_k(X,\rho)}{\la_k(Y,\rho)} \le C
$$
where $C$ is a constant depending only on $\Lambda$.
Since $\dim L^2(Y)=N$, we have $\la_k(Y,\rho)=\infty$
for all $k>N$.
Hence for $k=N+1$ the second alternative above cannot occur
unless $\la_k(X,\rho)=\infty$.
We conclude that $\la_{N+1}(X,\rho)>C^{-1}\rho^{-2}$.
Therefore for $c=C^{-1}$ we have  $\#_X^\rho(c\rho^{-2})\le N$.
\end{proof}

\begin{theorem}\label{t:many-eigenvalues}
Let $X=(X,d,\mu)$ be a mm-space whose measure has full support.
Let $\mu^\rho$ be the measure defined by \eqref{e:murho}.
Let $r\ge \rho$ and $N=N_X(3r)$. Then
$$
 \la_N(x,\rho) \le 4 Q(r) r^{-2}
$$
where
$$
 Q(r) = \sup_{x\in X}\frac{\mu^\rho(B_{2r}(x))}{\mu^\rho(B_{r/2}(x))} .
$$
\end{theorem}

For spaces satisfying reasonable assumptions,
Theorem \ref{t:many-eigenvalues} complements Theorem \ref{t:ess-spectrum}
by giving a lower bound on $\#_X^\rho(c\rho^{-2})$
of the same order of magnitude as in Theorem~\ref{t:ess-spectrum}.
Indeed, let $c$ be the constant from Theorem~\ref{t:ess-spectrum}
and assume that $Q(r)\le Q_{\max}$ for all $r\ge\rho$.
Then, applying Theorem \ref{t:many-eigenvalues} 
to $r=2\sqrt{c^{-1}Q_{\max}}\,\rho$ we get
$$
 \#_X^\rho(c\rho^{-2}) = 
 \#_X^\rho(4Q_{\max}r^{-2}) \ge N_X(3r)
 = N_X(C_1\rho)
$$
where $C_1=6\sqrt{c^{-1}Q_{\max}}$.

\begin{proof}[\bf Proof of Theorem \ref{t:many-eigenvalues}]
By the min-max formula \eqref{e:minmax}, it suffices to construct a
linear subspace $H\subset L^2(X)$ such that $\dim H=N$ and
\be\label{e:many-eigenvalues1}
D^\rho_X(u)\le 4Q(r)r^{-2}\|u\|^2_{L^2(X,\mu^\rho)}
\ee
for every $u\in H$.
Here $D^\rho_X$ is the Dirichlet form given by \eqref{e:dirichlet}.

Let $\{x_1,\dots,x_N\}$ be a $3r$-separated set in $X$.
For each $i$, define a function $u_i\co X\to\R$ by
$$
 u_i(x)= \max\left\{ 1- \frac{d(x, x_i)}{r} , 0 \right\} .
$$
Let $H$ be the linear span of $u_1,\dots,u_N$.
We are going to show that \eqref{e:many-eigenvalues1}
is satisfied for all $u\in H$.
The supports of $u_i$'s are separated
by distance at least $\rho$.
Hence $u_i\perp u_j$ and $\Dr_X(u_i)\perp u_j$
in $L^2(X,\mu^\rho)$  for all $i\ne j$.
Therefore it suffices to verify \eqref{e:many-eigenvalues1}
for $u=u_i$ only.

\smallbreak
Since $u_i(x)\ge \frac12$ for all $x\in B_{r/2}(x_i)$, we have
\begin{align*}
\|u_i\|^2_{L^2(X,\mu^\rho)} &= \rho^2 \int_X \mu(B_\rho(x)) u^2_i(x) d \mu(x) \\ 
&\geq \frac{\rho^2}{4} \int_{B_{r/2}(x_i)} \mu(B_\rho(x)) d\mu(x)
= \frac{\rho^2}{4}\mu^\rho(B_{r/2}(x)) .
\end{align*}
Since $u_i(x)=0$ if $x\notin B_r(x_i)$ and
$u_i$ is $(1/r)$-Lipschitz, we have
\begin{align*}
D^\rho_X(u_i) &= \int_{B_{r+\rho}(x_i)} \int_{B_\rho(x)} |u_i(x)-u_i(y)|^2 \,  d\mu(y)d\mu(x) \\ 
 &\leq \frac{\rho^2}{r^2} \int_{B_{r+\rho}(x_i)} \mu(B_\rho(x)) \, d\mu(x) 
 = \frac{\rho^2}{r^2} \mu^\rho(B_{r+\rho}(x_i)).
\end{align*}
Thus
$$
 \frac{D^\rho_X(u_i)}{\|u_i\|^2_{L^2(X,\mu^\rho)}}
 \le \frac 4{r^2} \frac{\mu^\rho(B_{r+\rho}(x_i))}{\mu^\rho(B_{r/2}(x))} 
 \le 4r^{-2} Q(r) .
$$
The theorem follows.
\end{proof}

\end{document}